\documentclass[11 pt]{article}

\usepackage[margin=1in]{geometry}
\usepackage{amsfonts,amsthm, amsmath,amssymb,mathrsfs}
\usepackage{ragged2e}
\usepackage{cite}
\usepackage[labelsep=period]{caption}
\usepackage{blindtext}
\usepackage{changepage}
\usepackage{float}
\usepackage{multirow}
\usepackage{rotating}
\usepackage{lscape}

\usepackage{apacite}

\renewcommand\neq{\mathrel{\vphantom{|}\mathpalette\xsneq\relax}}
\newcommand\xsneq[2]{%
  \ooalign{\hidewidth$#1|$\hidewidth\cr$#1=$\cr}%
}
\numberwithin{equation}{section}
\newtheorem{theorem}{Theorem}[section]
\newtheorem{lemma}{Lemma}[section]
\theoremstyle{plain}
\newtheorem{definition}[theorem]{Definition}
\newtheorem{problem}{Problem}[section]

\usepackage{subcaption}
\usepackage{graphicx}

\begin{document}
\title{\textbf{Comparison of fractional-order generalized wavelets and orthonormal wavelets methods for  solving multi-dimensional fractional optimal control problems } }
\author{Akanksha Singh*, S. Saha Ray, \\
\textit{Department of Mathematics}\\
\textit{National Institute of Technology Rourkela}\\
\textit{Rourkela-769008, India}\\
\textit{*19singhac@gmail.com}}

\maketitle

\begin{abstract} 
This paper presents an efficient numerical technique for solving multi-dimensional fractional optimal control problems using fractional-order generalized Bernoulli wavelets. The numerical results obtained by this method have been compared with the results obtained by the method using orthonormal Bernoulli wavelets. Using fractional-order generalized Bernoulli wavelets, product and integral operational matrices have been obtained. By using these operational matrices, the multi-dimensional fractional optimal control problems have been reduced into a system of algebraic equations. To confirm the efficiency of the proposed numerical technique involving fractional-order generalized Bernoulli wavelets, numerical problems have been solved by using both orthonormal Bernoulli wavelets and fractional-order generalized Bernoulli wavelets and obtaining an approximate cost function value, which has been found by approximating state and control functions. In addition, the convergence rate and error bound of the proposed numerical method have also been derived.

\end{abstract}
\numberwithin{equation}{section}
    
\section{Introduction}
Fractional-order calculus plays an important role in thermodynamics
\cite{meilanov2014thermodynamics}, viscoelasticity \cite{koeller1984applications}, rheology \cite{bagley1983theoretical}, viscoelastic and viscous-viscoelastic dampers \cite{ray2016formulation}, electrical circuits theory \cite{sikora2017fractional}, mechatronics \cite{morar2023cascade}, signal processing \cite{chen2007overview}, control theory \cite{matuvsuu2011application}, chaos theory \cite{ odibat2017chaos}, biological system \cite{rihan2013numerical} and other applications \cite{tarasova2016elasticity}.
 Optimal control is the process of determining control and state functions for a dynamic system over a period of time to minimize or maximize a cost function. Despite the fact that optimal control theory has been studied for many years, fractional optimal control theory is a relatively new branch of mathematics. Various definitions of fractional derivatives with singular and non-singular kernels can be applied to define a multi-dimensional fractional optimal control problems (FOCPs). Depending on the characteristics of the system being modelled, either a singular or non-singular kernel can be used to define fractional derivatives\cite{tuan2020numerical,odibat2023new,siddique2023analysis}
. Singular kernels are useful for systems with long-range interactions and memory effects, while non-singular kernels are useful for systems with short-range interactions and smooth dynamics.
 In the 21st century, optimal control has applications in robotics
, finance, and biomedical engineering \cite{mashayekhi2018approximate,agrawal1989general,agrawal2004general,
sabermahani2021fibonacci}. The capability of optimal
 control theory has been further improved by developing advanced 
computing tools and optimization algorithms \cite{heydari2018new,ghanbari2022numerical,barikbin2020solving}. In  reality, the
 behavior of many physical systems is governed more accurately by
  fractional order dynamics instead of integer order ones  \cite{herrmann2011fractional}. Various numerical and analytical methods have been used in recent years to solve different kinds of FOCPs \cite{bhrawy2017solving,dehestani2020fractional,
  ghaderi2021solving,sahu2018comparison}. Wavelets are extremely effective methods that are applied in a variety of numerical techniques. Most applications of the wavelet concepts are found in mathematics and engineering. Wavelets have been shown to be useful in a wide range of contexts, and they are especially used in signal analysis \cite{saha2020numerical,behera2022two,ray2019novel,ray2023two}. 
The operational matrix of fractional integration, the fractional Bernoulli wavelets basis, function approximation, and the Lagrange multiplier approach have been used in this study to solve a specific multi-dimensional FOCP. A typical multi-dimensional FOCP has been written as follows:

\begin{align}
 \min\Tilde{J}  =  \int_{0}^{1}{\mathcal{F}}({\zeta}, x_{1}({\zeta}), x_{2}({\zeta}),  u({\zeta}))d{\zeta},\label{1.1}
\end{align}
\begin{align}  
\mathcal{G}_{1}\left({^{C}}\mathcal{D}_{{\zeta}}^{{\mu}}x_{1}({\zeta}),{^{C}}\mathcal{D}_{{\zeta}}^{{\mu}}x_{2}({\zeta}) \right)= \mathcal{H}_{1}({\zeta}, x_{1}({\zeta}), x_{2}({\zeta}), u({\zeta})),\label{1.2}
\end{align}
\begin{align}  
\mathcal{G}_{2}\left({^{C}}\mathcal{D}_{{\zeta}}^{{\mu}}x_{1}({\zeta}),{^{C}}\mathcal{D}_{{\zeta}}^{{\mu}}x_{2}({\zeta}) \right)= \mathcal{H}_{2}({\zeta}, x_{1}({\zeta}), x_{2}({\zeta}), u({\zeta})),\label{1.3}
\end{align}
\begin{align}
    x_{1}(0)= x_{1},  x_{2}(0)= x_{2}.\label{1.4}
\end{align}
With this proposed method, the original multi-dimensional FOCP has been converted into a set of linear equations. The fractional derivative of state and control functions have been expanded using the fractional Bernoulli wavelets basis. Finally, the constrained extremum technique is applied, which usually  connecting the constraint equation derived from the given dynamical system to the cost function using a set of unknown Lagrange multipliers. Furthermore, the error estimation and convergence analysis of the proposed numerical technique have been established. 
\\

The contents of this paper are as follows: Section 2 discusses the definitions of fractional derivatives and integrals. Wavelets   are discussed in Section 3. Section 4 describes function approximation. Section 5 illustrates the operational matrix of integration for the fractional Bernoulli wavelet. In Section 6, a numerical approach has been proposed for solving the FOCP. Section 7 explains the estimation of the error. In Section 8, the convergence of the proposed fractional Bernoulli wavelets method has been established. In Section 9, numerical experimental results show the accuracy and efficiency of the proposed numerical scheme. Finally, Section 10 comes to an end with a concluding remark.

\section{ Preliminaries }
In this section, some important definitions of fractional calculus are provided.

\begin{definition}[Riemann-Liouville integral]
The Riemann-Liouville fractional integral of a function $\Tilde{\mathcal{Z}}( \sigma)$ of order $\mu > 0$ is defined  \textnormal{\cite{oldham20061974}} as follows:
\begin{equation}
 {_0}{}\mathcal{I}_{{\sigma}}^{\mu} \Tilde{\mathcal{Z}}( \sigma) =\dfrac{1}{\Gamma(\mu)}  \int_{0}^{\sigma} (\sigma-\varpi)^{\mu-1}\Tilde{\mathcal{Z}}(\varpi)d\varpi, \ \ \ \ \     \mu> 0, \sigma > 0. \label{2.1}
\end{equation}

\end{definition}

\begin{definition}[Caputo fractional derivative ]
Caputo fractional derivative was introduced by M. Caputo in 1967. The Caputo fractional derivative of a function $\Tilde{\mathcal{Z}}({\sigma})$ of order $\mu$ is defined \textnormal{ \cite{siddique2023analysis}} as follows:

\begin{equation}
{_{\sigma_{0}}^C}\mathcal{D}_{\sigma}^{\mu}\Tilde{\mathcal{Z}}(\sigma)=
\begin{cases} 
\dfrac{1}{\Gamma(1-\mu)}\displaystyle\int_{\sigma_{0}}^{\sigma}{(\sigma-\varpi)^{-\mu}}\frac{d\Tilde{\mathcal{Z}}(\varpi)}{d\varpi}d\varpi, \ \ & 0<\mu<1, \\  
\mathcal{Z}'(\sigma), & \mu=1.
\end{cases}\label{2.2} 
\end{equation}
\end{definition}

\section{Wavelets}
This section gives an overview of orthonormal Bernoulli wavelets and fractional Bernoulli wavelets.
\subsection{Orthonormal Bernoulli wavelets}

Let $k$ be a positive integer. For each $n =1, 2,  \dots, 2^{k-1}$ and non-negative integer $m$, the orthonormal Bernoulli wavelets (OBWs) $\psi_{n, m}(\zeta)$, are defined on the interval $[0, 1)$ as
\begin{equation}
\psi_{n,m}(\zeta)=
\begin{cases}

2^{\frac{k-1}{2}} \Tilde{\mathcal{B}}_{m}(2^{\frac{k-1}{2}}\zeta-n+1),  \quad if \quad \dfrac{n-1}{2^{k-1}}\le \zeta <\dfrac{n}{2^{k-1}},\\
0, \quad  \quad  \quad  \quad \quad \quad  \quad  \quad \quad \quad  \quad  \quad \quad otherwise, \label{3.1}
\end{cases}
\end{equation}
where $\Tilde{\mathcal{B}}_{0}(\zeta)=1$ and\\
 $$\Tilde{\mathcal{B}}_{m}({\zeta})=\dfrac{B_{m}({\zeta})}{\Upsilon_{m} },\quad  m>0,$$
 and $\Upsilon_{m}= \sqrt{\dfrac{(-1)^{m-1}(m!)^2}{(2m)!}B_{2m}(0)}$ is the normality coefficient. $B_{m}({\zeta}), m=0,1,2,\dots,M-1$ are known as Bernoulli polynomial and given by
 $$B_{m}(\zeta)=\sum_{j=0}^{m} \binom{m}{j}B_{m-j}(0){\zeta}^{j},$$
 where $ B_{j}(0)$'s are Bernoulli numbers. Therefore, Bernoulli wavelets for $ m > 0$ can be rewritten as
\begin{align}
\psi_{{n},{m}}(\zeta) = {\xi}_{m} \sum_{j=0}^{m}\binom{m}{j}B_{m-j}(0) 2^{(k-1)j}\left({\zeta}-\dfrac{n-1}{2^{k-1}}\right)^{j} \chi({\zeta})\Big\vert_{{\zeta}\in \left[\frac{{n}-1}{2^{{k}-1}}, \frac{{n}}{2^{{k}-1}}\right)},\label{3.2}
\end{align} 
where ${\xi}_{m}=\sqrt{\dfrac{2^{k-1}(2m)!}{(-1)^{m-1}(m!)^2B_{2m}(0)}}$ and
$\chi({\zeta})\Big\vert_{{\zeta}\in \left[\frac{{n}-1}{2^{{k}-1}}, \frac{{n}}{2^{{k}-1}}\right)}$  is the characteristic  function defined as follows:
\begin{equation}
\chi({\zeta})\Big\vert_{{\zeta}\in \left[\frac{{n}-1}{2^{{k}-1}}, \frac{{n}}{2^{{k}-1}}\right)}=
\begin{cases}
1, ~ ~ ~ ~ ~ ~{\zeta}\in \left[\frac{{n}-1}{2^{{k}-1}}, \frac{{n}}{2^{{k}-1}}\right),\\
0, ~ ~ ~ ~ ~ ~ ~ ~ ~ ~ otherwise.\label{3.3}
\end{cases}
\end{equation}
When $m=0$, Eq. \eqref{3.2} becomes $$\psi_{{n},{0}}({\zeta})=2^{\left(\frac{k-1}{2}\right)} \chi({\zeta})\Big\vert_{{\zeta}\in \left[\frac{{n}-1}{2^{{k}-1}}, \frac{{n}}{2^{{k}-1}}\right)}.$$
\subsection{Fractional Bernoulli wavelets}
By applying the transformation ${\zeta}={x}^{\mu}$ to Eq. \eqref{3.1} for a positive real number ${\mu}$, the fractional Bernoulli wavelets (FBWs) are defined \cite{nosrati2021using} as follows:
\begin{align}
\psi_{n,m}^{{\mu}}({x})=
\begin{cases}
2^{\frac{k-1}{2}} \Tilde{\mathcal{B}}_{m}(2^{\frac{k-1}{2}}{x}^{{\mu}}-n+1),  \quad if \quad \left(\dfrac{n-1}{2^{k-1}}\right)^{\frac{1}{{\mu}}}\le {x} <\left(\dfrac{n}{2^{k-1}}\right)^{\frac{1}{{\mu}}},\\
0, \quad  \quad  \quad  \quad \quad \quad  \quad  \quad \quad \quad  \quad  \quad \quad otherwise. \label{3.4}
\end{cases}
\end{align}
By using  Eq. \eqref{3.2}, FBWs are defined as
\begin{align}
\psi_{{n},{m}}^{{\mu}}({\zeta}) = {\xi}_{m} \sum_{j=0}^{m}\binom{m}{j}B_{m-j}(0) 2^{(k-1)j}\left({\zeta}^{{\mu}}-\dfrac{n-1}{2^{k-1}}\right)^{j} \chi({\zeta})\Big\vert_{{\zeta}\in  \left[\left(\frac{{n}-1}{2^{{k}-1}}\right)^\frac{1}{{\mu}}, \left(\frac{{n}}{2^{{k}-1}}\right)^\frac{1}{{\mu}}\right)}.\label{3.5}
\end{align}

\section{Function approximation and product operational matrix}
In this section, function approximation and product operational matrix of FBWs are given below.
\subsection{Function approximation }
A function $f({\zeta})$, square integrable in $[0,1]$,  can be expressed in terms of the wavelet as
\begin{equation}
f({\zeta})\approx\sum_{{n}=1}^{2^{{k}-1}}\sum_{{m}=0}^{M-1}\Hat{c}_{{n} {m}}\psi_{{n}, {m}}^{{\mu}}({\zeta})=\Hat{C}^{{T}}\Hat{\Psi}^{{\mu}}({\zeta}), \label{4.1}
\end{equation}
where $\Hat{C}$ and $\Hat{\Psi}^{{\mu}}({\zeta})$ are column vectors of dimension $(2^{{k}-1}M\times1)$ given by
\begin{align}\nonumber
\Hat{C}^{T}=&[\Hat{c}_{1 0}, \Hat{c}_{1 1},  \dots \Hat{c}_{1 M-1},  \dots, \Hat{c}_{2^{{k}-1} M-1}]\\
 =&[\Hat{c}_{1}, \Hat{c}_{2},  \dots, \Hat{c}_{\Hat{m}}],\label{4.2}
\end{align}
and
\begin{align}\nonumber
 \Hat{\Psi}^{{\mu}}({\zeta}) &=[\psi_{1, 0}^{{\mu}}({\zeta}), \psi_{1, 1}^{{\mu}}({\zeta}), \dots, \psi_{1, M-1}^{{\mu}}({\zeta}), \dots, \psi_{2^{{k}-1}, M-1}^{{\mu}}({\zeta})] \\
&=[\Psi_{1}^{{\mu}}({\zeta}), \Psi_{ 2}^{{\mu}}({\zeta}),  \dots,  \Psi_{\Hat{m}}^{{\mu}}({\zeta})],\label{4.3}
\end{align} 
where $\Hat{m}=2^{{k}-1}M.$\\
The coefficient vector $\Hat{C}$ can be obtained from Eq. \eqref{4.1} as follows:
\begin{equation}
\Hat{C}^{T}=\Bigg(\int_0^1f({\zeta})\left(\Hat{\Psi}^{{\mu}}({\zeta})\right)^Td{\zeta}\Bigg) \Hat{D}^{-1}({\mu}),\label{4.4}
\end{equation}
where $\Hat{D}(\mu)$ is a square matrix of order $2^{{k}-1}M$ and is given by
\begin{equation}
\Hat{D}({\mu})=\int_0^1 \Hat{\Psi}^{{\mu}}({\zeta})(\Hat{\Psi}^{{\mu}}({\zeta}))^{{T}}d{\zeta}. \label{4.5}
\end{equation}
For FBWs, when $k = 2$ and $M = 3$,
$$
\Hat{D}(0.9)=
\begin{bmatrix}
0.925875&	0.0844033&	-0.0311326& 0& 0& 0\\
0.0844033&	0.898029	&0.0903579& 0& 0& 0\\
-0.0311326	&0.0903579	&0.896687& 0& 0& 0\\
0& 0& 0&	1.07413&	0.0234615&	-0.00184153\\
0& 0& 0&	0.0234615	&1.07248	&0.0211615\\
0& 0& 0&	-0.00184153	&0.0211615	&1.07293\\

\\
\end{bmatrix}.
\quad
$$
When ${\mu}=1 $,  FBWs changes into OBWs as
$$
\Hat{D}(1)=
\begin{bmatrix}
1&0	&0 &	0 &	0	&0\\
0&1	&0	&0	&0&	0\\
0&0	&1	&0	&0	&0\\
0&0	&0	&1	&0	&0\\
0&0	&0	&0	&1	&0\\
0&0	&0	&0	&0	&1\\
\end{bmatrix}.
\quad
$$

\subsection{Product operational matrix}
The following property of the product of two FBWs function vectors will also be used:
\begin{align}
  \Hat{\Psi}^{{\mu}}({\zeta}) (\Hat{\Psi}^{{\mu}}({\zeta}))^{{T}}\Hat{C}^{{T}} \approx \Tilde{C}\Hat{\Psi}^{{\mu}}({\zeta}),  
\end{align}
where $ \Tilde{C}$ is a $ 2^{k-1}M  \times 2^{k-1}M $ product operational matrix given as
\begin{align*}
\Tilde{C}=\left(\int_{0}^{1}\Hat{\Psi}^{{\mu}}({\zeta}) (\Hat{\Psi}^{{\mu}}({\zeta}))^{{T}}\Hat{C}^{{T}}(\Hat{\Psi}^{{\mu}}({\zeta}))^{{T}}d\zeta\right)\Hat{D}^{-1}({\mu}).
\end{align*}
\section{Fractional Bernoulli wavelets operational matrix}
In the present analysis, the operational matrix for fractional-order integral has been derived. 
\begin{theorem}
Let $ \Hat{\Psi}^{{\mu}}({{\zeta}})$ be the FBWs vector introduced in Eq. \eqref{4.3}. Then the fractional-order integral operational matrix is given by
$$\mathcal{I}^{\mu}(\Hat{\Psi}^{{\mu}}({\zeta}))\approx \mathcal{P}^{\mu}\Hat{\Psi}^{{\mu}}({{\zeta}}).$$
Also, the $(n,m)$-th element of fractional-order integral of vector $\Hat{\Psi}^{{\mu}}({{\zeta}})$ is given by
$$\mathcal{I}^{\mu} (\psi_{{n},{m}}^{{\mu}}({\zeta}))  \approx \sum_{r=1}^{2^{{{k}}-1}}\sum_{l=0}^{{M}-1}\Theta_{r l}^{{\mu}; {n} {{m}}}\psi_{{r}, {l}}^{{\mu}}({{\zeta}}), ~ ~ ~ ~ ~ ~ ~ ~ ~n=1,2,\dots,2^{{k}-1},\quad{{m}} = 0, 1, 2,  \dots, {M}-1,$$
where
\begin{equation*}
\Theta_{r l}^{{\mu}; {n} {{m}}}=
\begin{cases}
\displaystyle \xi_{m}\sum_{j=0}^{m} \binom{m}{j}B_{m-j}(0) 2^{({k}-1)j} c_{rl}^{1{{j}}}, \quad  \quad \quad \quad \quad \quad \quad \quad \quad \quad \quad \quad  \quad n=1,\\
\displaystyle \xi_{m}\sum_{j=0}^{m} \sum_{q=0}^{j} \binom{m}{j}\binom{j}{q}B_{m-j}(0) 2^{({k}-1)q}(-1)^{j-q}(n-1)^{j-q} c_{rl}^{n{{q}}}, \quad \quad n\neq 1,
\end{cases}
\end{equation*}
and $\xi_{m}=\sqrt{\dfrac{2^{k-1}(2m)!}{(-1)^{m-1}(m!)^2B_{2m}(0)}} $.
\end{theorem}
\begin{proof}
The fractional  integration of the FBWs vector $\Hat{\Psi}^{{\mu}}({{\zeta}})$ can be determined as
\begin{equation}
\mathcal{I}^{\mu}(\Hat{\Psi}^{{\mu}}({{\zeta}}))\approx \mathcal{P}^{\mu}\Hat{\Psi}^{{\mu}}({{\zeta}}),\label{5.1}
\end{equation}
where $\mathcal{P}^{\mu}$ is operational matrix of order $2^{{{k}}-1}{M}$. 
Using Eq. \eqref{3.5}, FBWs can be written as follows:
\begin{align}
\psi_{{n},{m}}^{{\mu}}({\zeta}) = {\xi}_{m} \sum_{j=0}^{m}\binom{m}{j}B_{m-j}(0) 2^{(k-1)j}\left({\zeta}^{{\mu}}-\dfrac{n-1}{2^{k-1}}\right)^{j} \chi({\zeta})\Big\vert_{{\zeta}\in  \left[\left(\frac{{n}-1}{2^{{k}-1}}\right)^\frac{1}{{\mu}}, \left(\frac{{n}}{2^{{k}-1}}\right)^\frac{1}{{\mu}}\right)},\label{5.2}
\end{align}
where  $${\xi}_{m}=\sqrt{\dfrac{2^{k-1}(2m)!}{(-1)^{m-1}(m!)^2B_{2m}(0)}},$$ and~ $\chi({\zeta})\Big\vert_{{\zeta}\in \left[\left(\frac{{n}-1}{2^{{k}-1}}\right)^\frac{1}{{\mu}}, \left(\frac{{n}}{2^{{k}-1}}\right)^\frac{1}{{\mu}}\right)}$  is the characteristic  function defined as follows:
\begin{equation}
\chi({\zeta})\Big\vert_{{\zeta}\in \left[\left(\frac{{n}-1}{2^{{k}-1}}\right)^\frac{1}{{\mu}}, \left(\frac{{n}}{2^{{k}-1}}\right)^\frac{1}{{\mu}}\right)}=
\begin{cases}
1, ~ ~ ~ ~ ~ ~{\zeta}\in \left[\left(\frac{{n}-1}{2^{{k}-1}}\right)^\frac{1}{{\mu}}, \left(\frac{{n}}{2^{{k}-1}}\right)^\frac{1}{{\mu}}\right),\\
0, ~ ~ ~ ~ ~ ~ ~ ~ ~ ~ otherwise.\label{5.3}
\end{cases}  
\end{equation}
When ${n}=1$,  Eq. \eqref{5.2} can be written  as follows:
\begin{align}
\psi_{{1},{m}}^{{\mu}}({\zeta}) = {\xi}_{m} \sum_{j=0}^{m}\binom{m}{j}B_{m-j}(0) 2^{(k-1)j}\left({\zeta}^{{\mu} j}\right) \chi({\zeta})\Big\vert_{{\zeta}\in  \left[\left(\frac{{n}-1}{2^{{k}-1}}\right)^\frac{1}{{\mu}}, \left(\frac{{n}}{2^{{k}-1}}\right)^\frac{1}{{\mu}}\right)},\label{5.4}
\end{align}
\begin{align}
\mathcal{I}^{\mu}\left(\psi_{1, {m}}^{{\mu}} ({\zeta})\right) ={\xi}_{m} \sum_{j=0}^{m}\binom{m}{j}B_{m-j}(0) 2^{(k-1)j} \mathcal{I}^{\mu}\left({\zeta}^{{\mu} j} \chi({\zeta})\Big\vert_{{\zeta}\in \left[\left(\frac{{n}-1}{2^{{k}-1}}\right)^{\frac{1}{\mu}}, \left( \frac{{n}}{2^{{k}-1}}\right)^{\frac{1}{\mu}}\right)}\right). \label{5.5}
\end{align}
When ${n} = 2, 3, 4,  \dots,2^{{k}-1}$, 
Eq. \eqref{5.2} can be written  as follows:
\begin{align}
\psi_{{n},{m}}^{{\mu}}({\zeta}) = \displaystyle \xi_{m}\sum_{j=0}^{m} \sum_{q=0}^{j} \binom{m}{j}\binom{j}{q}B_{m-j}(0) 2^{({k}-1)q}(-1)^{j-q}(n-1)^{j-q}\left({\zeta}^{{\mu} q}\right) \chi({\zeta})\Big\vert_{{\zeta}\in  \left[\left(\frac{{n}-1}{2^{{k}-1}}\right)^\frac{1}{{\mu}}, \left(\frac{{n}}{2^{{k}-1}}\right)^\frac{1}{{\mu}}\right)},\label{5.6}
\end{align}
\begin{align}
\mathcal{I}^{\mu}(\psi_{{n}, {m}}^{{\mu}} ({\zeta})) =\displaystyle \xi_{m}\sum_{j=0}^{m} \sum_{q=0}^{j} \binom{m}{j}\binom{j}{q}B_{m-j}(0) 2^{({k}-1)q}(-1)^{j-q}(n-1)^{j-q}\mathcal{I}^{\mu}\left({\zeta}^{{\mu} q} \chi({\zeta})\Big\vert_{{\zeta}\in  \left[\left(\frac{{n}-1}{2^{{k}-1}}\right)^\frac{1}{{\mu}}, \left(\frac{{n}}{2^{{k}-1}}\right)^\frac{1}{{\mu}}\right)}\right).\label{5.7}
\end{align}
Now, approximating $\mathcal{I}^{\mu}\left({\zeta}^{{\mu} q}\chi({\zeta})\Big\vert_{{\zeta}\in \left[\left(\frac{{n}-1}{2^{{k}-1}}\right)^\frac{1}{{\mu}}, \left(\frac{{n}}{2^{{k}-1}}\right)^\frac{1}{{\mu}}\right)}\right)$ by $2^{{k}-1}M$ terms of FBWs, yields
\begin{equation}
\mathcal{I}^{\mu}  \left({\zeta}^{{\mu} q}\chi({\zeta})\Big\vert_{{\zeta}\in \left[\left(\frac{{n}-1}{2^{{k}-1}}\right)^\frac{1}{{\mu}}, \left(\frac{{n}}{2^{{k}-1}}\right)^\frac{1}{{\mu}}\right)}\right)=g_{{n}q}({\zeta})\approx \sum_{r=1}^{2^{{k}-1}}\sum_{l=0}^{M-1}c_{rl}^{nq}\psi_{r,l}^{{\mu}}({\zeta})=C_{rl}^{T}\Hat{\Psi}^{{\mu}}({\zeta}),  \label{5.8}
\end{equation}
where $C_{rl}^{T} =<g_{{n}q}({\zeta}),  (\Hat{\Psi}^{{\mu}}({\zeta}))^T>\Hat{D}^{-1}({\mu}).$\\\\
Substituting Eq. \eqref{5.8} into Eqs. \eqref{5.5} and \eqref{5.7}, yields 
\begin{align}\nonumber
\mathcal{I}^{\mu}(\psi_{1,{m}}^{{\mu}} ({\zeta})) &\approx {\xi}_{m} \sum_{j=0}^{m}\binom{m}{j}B_{m-j}(0) 2^{(k-1)j}\sum_{r=1}^{2^{{k}-1}}\sum_{l=0}^{M-1}c_{rl}^{1j}\psi_{r,l}^{{\mu}}({\zeta})\\
&=\sum_{r=1}^{2^{{k}-1}}\sum_{l=0}^{M-1}\Theta_{r l}^{{\mu}; 1 {m}}\psi_{{r}, {l}}^{{\mu}}({\zeta}), \quad \quad \quad \quad 
 {m} = 0, 1, 2,  \dots, M-1,\label{5.9}
\end{align}
where
$\Theta_{r l}^{{\mu}; 1 {m}}= \displaystyle{\xi}_{m} \sum_{j=0}^{m}\binom{m}{j}B_{m-j}(0) 2^{(k-1)j} c_{rl}^{1{j}}$.\\
\begin{align}\nonumber
\mathcal{I}^{\mu}(\psi_{n,{m}}^{{\mu}} ({\zeta})) &\approx \displaystyle \xi_{m}\sum_{j=0}^{m} \sum_{q=0}^{j} \binom{m}{j}\binom{j}{q}B_{m-j}(0) 2^{({k}-1)q}(-1)^{j-q}(n-1)^{j-q}\sum_{r=1}^{2^{{k}-1}}\sum_{l=0}^{M-1}c_{rl}^{nq}\psi_{r,l}^{{\mu}}({\zeta})\\
&=\sum_{r=1}^{2^{{k}-1}}\sum_{l=0}^{M-1}\Theta_{r l}^{{\mu}; n {m}}\psi_{{r}, {l}}^{{\mu}}({\zeta}), \quad \quad \quad \quad 
 n=2,3,\dots,2^{k-1},{m} = 0, 1, 2,  \dots, M-1,\label{5.10}
\end{align}
where
$\Theta_{r l}^{{\mu}; n {m}}= \displaystyle \xi_{m}\sum_{j=0}^{m} \sum_{q=0}^{j} \binom{m}{j}\binom{j}{q}B_{m-j}(0) 2^{({k}-1)q}(-1)^{j-q}(n-1)^{j-q}c_{rl}^{n{q}}$.\\
Therefore,
\begin{align*}
\mathcal{I}^{\mu} (\psi_{{n},{m}}^{{\mu}}({\zeta})) &\approx \sum_{r=1}^{2^{{{k}}-1}}\sum_{l=0}^{{M}-1}\Theta_{r l}^{{\mu}; {n} {{m}}}\psi_{{r}, {l}}^{{\mu}}({\zeta}) ~ ~ ~ ~ ~ ~ ~ ~ ~n=1,2,\dots,2^{{k}-1},\quad{{m}} = 0, 1, 2,  \dots, {M}-1,\\
&= \mathcal{P}^{\mu}\Hat{\Psi}^{{\mu}}({\zeta}),
\end{align*}
where
$$ \mathcal{P}^{\mu} = 
\begin{bmatrix}  
\Theta_{1 0}^{{\mu}; 1 0} &  \dotsb  &\Theta_{1 M-1}^{{\mu}; 1 0} & \dotsb& \Theta_{2^{{k}-1} 0}^{{\mu}; 1 0}& \dotsb     &     \Theta_{2^{{k}-1} M-1}^{{\mu}; 1 0}\\
 \Theta_{1 0}^{{\mu}; 1 1} & \dotsb &\Theta_{1 M-1}^{{\mu}; 1 1} & \dotsb & \Theta_{2^{{k}-1}0}^{{\mu}; 1 1}& \dotsb    &     \Theta_{2^{{k}-1} M-1}^{{\mu}; 1 1}\\
  &  &         &   &  &  \\ \\
\vdots  & \dotsb  &   \vdots       & \dotsb   & \vdots  & \ddots & \vdots\\
  &  &         &   &  &  \\
\Theta_{1 0}^{{\mu}; 2^{{k}-1} M-1} & \dotsb &\Theta_{1 M-1}^{{\mu}; 2^{{k}-1} M-1} &\dotsb & \Theta_{2^{{k}-1} 0}^{{\mu}; 2^{{k}-1} M-1}& \dotsb     &     \Theta_{2^{{k}-1} M-1}^{{\mu}; 2^{{k}-1} M-1}\\ 
\quad
\end{bmatrix}.
\quad
$$  
\end{proof}
Corollary: When $ k=2,M=2$ and ${\mu}=0.9$, operational matrix for FBWs
$$
\mathcal{P}^{0.9}=
\begin{bmatrix}

0.197148	&0.122578 &	0.416644	&6.13304\times10^{-12}\\
-0.0998925 &	0.0111862 &0.038012&	5.59535\times10^{-13}\\
0&  0&	0.238628	&0.139667\\
0&	0& -0.134174&	0.00305066\\

\end{bmatrix}
\quad
.$$
\section{Numerical scheme }
Consider the multi-dimensional FOCP defined as follows: 
\begin{align}
 \min\Tilde{J}  =  \dfrac{1}{2}\int_{0}^{1}\left(a(\zeta) x_{1}^2({\zeta})+b(\zeta) x_{2}^2({\zeta})+c(\zeta)u^2({\zeta})\right)d{\zeta},\label{6.1}
\end{align}
\begin{align}  
\mathcal{G}_{1}\left({^{C}}\mathcal{D}_{{\zeta}}^{{\mu}}x_{1}({\zeta}),{^{C}}\mathcal{D}_{{\zeta}}^{{\mu}}x_{2}({\zeta}) \right)= \mathcal{H}_{1}({\zeta}, x_{1}({\zeta}), x_{2}({\zeta}), u({\zeta})),\label{6.2}
\end{align}
\begin{align}  
\mathcal{G}_{2}\left({^{C}}\mathcal{D}_{{\zeta}}^{{\mu}}x_{1}({\zeta}),{^{C}}\mathcal{D}_{{\zeta}}^{{\mu}}x_{2}({\zeta}) \right)= \mathcal{H}_{2}({\zeta}, x_{1}({\zeta}), x_{2}({\zeta}), u({\zeta})),\label{6.3}
\end{align}
\begin{align}
    x_{1}(0)= x_{1},x_{2}(0)= x_{2}.\label{6.4}
\end{align}
\\
First, expand the fractional derivative of the state functions  by the FBWs basis $\Hat{\Psi}^{{\mu}}({\zeta})$.
\begin{align}
 {^{C}}\mathcal{D}_{{\zeta}}^{{\mu}}x_{1}({\zeta})\approx \Hat{C}_{1}^{T}\Hat{\Psi}^{{\mu}}({\zeta}),\label{6.5}
\end{align}
\begin{align}
 {^{C}}\mathcal{D}_{{\zeta}}^{{\mu}}x_{2}({\zeta})\approx \Hat{C}_{2}^{T}\Hat{\Psi}^{{\mu}}({\zeta}),\label{6.6}
\end{align}
\begin{align}
      u({\zeta}) \approx \Hat{U}^{T}\Hat{\Psi}^{{\mu}}({\zeta}),\label{6.7}
\end{align}
where $\Hat{C}_{1}^{T}=[\Hat{c}_{11}, \Hat{c}_{12},\dots, \Hat{c}_{1\Hat{m}}]$, $\Hat{C}_{2}^{T}=[\Hat{c}_{21}, \Hat{c}_{22}, \dots, \Hat{c}_{2\Hat{m}}],$ and $
\Hat{U}^{T}=[\Hat{u}_{1}, \Hat{u}_{2}, \dots, \Hat{u}_{\Hat{m}}]$
are unknowns. \\
Also approximating $ a({\zeta}), b({\zeta}), c({\zeta}), x_{1}(0)$, and $ x_{2}(0)$ by the FBWs basis as 
\begin{align*}
a({\zeta})\approx A^T\Hat{\Psi}^{{\mu}}({\zeta}), b({\zeta})\approx B^T\Hat{\Psi}^{{\mu}}({\zeta}),\\
c({\zeta})\approx C^T\Hat{\Psi}^{{\mu}}(\zeta), x_{1}(0) \approx d_{1}^T\Hat{\Psi}^{{\mu}}({\zeta}),\\
x_{2}(0)\approx d_{2}^T\Hat{\Psi}^{{\mu}}({\zeta}),
\end{align*}
where
\begin{align*}
A^T&=[a_{1}, a_{2}, \dots, a_{\Hat{m}}],\quad B^T=[b_{1}, b_{2}, \dots, b_{\Hat{m}}],\\
C^T&=[c_{1}, c_{2}, \dots, c_{\Hat{m}}], \quad d_{1}^T=[d_{11}, d_{12}, \dots, d_{1\Hat{m}}],\\
d_{2}^T&=[d_{21}, d_{22}, \dots, d_{2\Hat{m}}],
\end{align*}
\begin{align*}
{a}_{l}=\Big(\int_0^1 a({\zeta})\left(\Hat{\Psi}^{{\mu}}({\zeta})\right)^{T}d{\zeta}\Big)\Hat{D}^{-1}(\mu),~~~~
 {b}_{l}=\Big(\int_0^1 b({\zeta})\left(\Hat{\Psi}^{{\mu}}({\zeta})\right)^{T}d{\zeta}\Big)\Hat{D}^{-1}({\mu}),\\
{c} _{l}=\Big(\int_0^1 c({\zeta})\left(\Hat{\Psi}^{{\mu}}({\zeta})\right)^{T}d{\zeta}\Big)\Hat{D}^{-1}({\mu}),~~~~ {d}_{1l}=\Big(\int_0^1 x_{1}(0)\left(\Hat{\Psi}^{{\mu}}({\zeta})\right)^{T}d{\zeta}\Big)\Hat{D}^{-1}({\mu}),\\
{d}_{2l}=\Big(\int_0^1 x_{2}(0)\left(\Hat{\Psi}^{{\mu}}({\zeta})\right)^{T}d{\zeta}\Big)\Hat{D}^{-1}({\mu}), \quad l=1,2, \dots, \Hat{m}.
\end{align*}  
Using  fractional operational matrix of integration, $x_{1}({\zeta})$ and $x_{2}({\zeta})$ can be represented as
\begin{align}\nonumber
\mathcal{I}^{{\mu}}\left({^{C}}{}\mathcal{D}_{{\zeta}}^{{\mu}}x_{1}({\zeta})\right)&=x_{1}({\zeta})-x_{1}(0),\\ \nonumber
x_{1}({\zeta})&=\mathcal{I}^{{\mu}}\left({^{C}}{}\mathcal{D}_{{\zeta}}^{{\mu}}x_{1}({\zeta})\right)+x_{1}(0)\\
&\approx \left(\Hat{C}_{1}^{T} \mathcal{P}^{{\mu}} + d_{1}^{T}\right) \Hat{\Psi}^{{\mu}}({\zeta}).\label{6.8}
\end{align}
\begin{align}\nonumber
\mathcal{I}^{{\mu}}\left({^{C}}{}\mathcal{D}_{{\zeta}}^{{\mu}}x_{2}({\zeta})\right)&=x_{2}({\zeta})-x_{2}(0),\\ \nonumber
x_{2}({\zeta})&=\mathcal{I}^{{\mu}}\left({^{C}}{}\mathcal{D}_{{\zeta}}^{{\mu}}x_{2}({\zeta})\right)+x_{2}(0)\\
&\approx \left(\Hat{C}_{2}^{T} \mathcal{P}^{{\mu}} + d_{2}^{T}\right) \Hat{\Psi}^{{\mu}}({\zeta}).\label{6.9}
\end{align}
where $\mathcal{P}^{\mu}$ is the fractional operational matrix of integration of order ${\mu}$. \\   
Now, approximating  $x_{1}^2(\zeta)$ by using FBWs, yields
\begin{align}\nonumber
x_{1}^2({\zeta}) &\approx\left(\left(\Hat{C}_{1}^{T}\mathcal{P}^{\mu}+d_{1}^{T}\right)\Hat{\Psi}^{\mu}({\zeta})\right)^2\\ \nonumber
&=\left(\Hat{C}_{1}^{T}\mathcal{P}^{\mu}+d_{1}^{T}\right)\Hat{\Psi}^{\mu}({\zeta})\left(\Hat{\Psi}^{\mu}({\zeta})\right)^{T}\left(\Hat{C}_{1}^{T}\mathcal{P}^{\mu}+d_{1}^{T}\right)^{T}\\
&=\Hat{C}_{12}^{T}\Hat{\Psi}^{\mu}({\zeta})\left(\Hat{\Psi}^{\mu}({\zeta})\right)^{T}\Hat{C}_{12}, \label{6.10} 
\end{align}
where $ \Hat{C}_{12}=(\Hat{C}_{1}^{T}\mathcal{P}^{\mu}+d_{1}^{T})^{T}$.\\
Now, approximating
\begin{align*}
\Hat{\Psi}^{\mu}({\zeta})(\Hat{\Psi}^{\mu}({\zeta}))^{T}\Hat{C}_{12}\approx\Hat{C}_{13}\Hat{\Psi}^{\mu}({\zeta}),
\end{align*}
 where $\Hat{C}_{13}$ is the product operational matrix of order $\Hat{m}$ and given as 
\begin{align*}
\Hat{C}_{13}=\Bigg( \int_{0}^{1} \left(\Hat{\Psi}^{\mu}({\zeta})\left(\Hat{\Psi}^{\mu}({\zeta})\right)^T\Hat{C}_{12}\left(\Hat{\Psi}^{\mu}({\zeta})\right)^{T}\right)d{\zeta}\Bigg)\Hat{D}^{-1}(\mu).
\end{align*}
Thus, from Eq. \eqref{6.10}, we obtain
\begin{align*}
x_{1}^2({\zeta}) &\approx\Hat{C}_{12}^{T}\Hat{C}_{13}\Hat{\Psi}^{\mu}({\zeta})\\
&=\Hat{C}_{14}^T\Hat{\Psi}^{\mu}({\zeta}),
\end{align*}
where $\Hat{C}_{14}^T=\Hat{C}_{12}^{T}\Hat{C}_{13}.$\\
Now,
\begin{align*}
a({\zeta})x_{1}^2({\zeta}) & \approx \left({A}^{T}\Hat{\Psi}^{\mu}({\zeta})\right)\left(\Hat{C}_{14}^{T}\Hat{\Psi}^{\mu}({\zeta})\right)^T\\
&={A}^{T}\Hat{\Psi}^{\mu}({\zeta})\left(\Hat{\Psi}^{\mu}({\zeta})\right)^T\Hat{C}_{14},
\end{align*}
Again,  approximating
\begin{align*}
\Hat{\Psi}^{\mu}({\zeta})(\Hat{\Psi}^{\mu}({\zeta}))^{T}\Hat{C}_{14} \approx \Hat{C}_{15}\Hat{\Psi}^{\mu}({\zeta}),
\end{align*}
where $\Hat{C}_{15}$ is product operational matrix of order $\Hat{m}$, given as
\begin{align*}
\Hat{C}_{15}=\Bigg(\int_{0}^{1} \left(\Hat{\Psi}^{\mu}({\zeta})\left(\Hat{\Psi}^{\mu}({\zeta})\right)^T\Hat{C}_{14}\left(\Hat{\Psi}^{\mu}({\zeta})\right)^{T}\right)d{\zeta}\Bigg)\Hat{D}^{-1}(\mu).
\end{align*}
Thus,
\begin{align*}
~~~~~~~~ a({\zeta})x_{1}^2({\zeta}) &  \approx A^{T}\Hat{C}_{15}\Hat{\Psi}^{\mu}({\zeta})\\
&=\Hat{C}_{16}^{T}\Hat{\Psi}^{\mu}({\zeta}),
\end{align*}
where $\Hat{C}_{16}^{T}= A^{T}\Hat{C}_{15}.$\\
Similarly, approximating $x_{2}^2({\zeta})$, yields
\begin{align}\nonumber
x_{2}^2({\zeta}) &\approx\left(\left(\Hat{C}_{2}^{T}\mathcal{P}^{\mu}+d_{2}^{T}\right)\Hat{\Psi}^{\mu}({\zeta})\right)^2\\ \nonumber
&=\left(\Hat{C}_{2}^{T}\mathcal{P}^{\mu}+d_{2}^{T}\right)\Hat{\Psi}^{\mu}({\zeta})\left(\Hat{\Psi}^{\mu}({\zeta})\right)^{T}\left(\Hat{C}_{2}^{T}\mathcal{P}^{\mu}+d_{2}^{T}\right)^{T}\\
&=\Hat{C}_{22}^{T}\Hat{\Psi}^{\mu}({\zeta})\left(\Hat{\Psi}^{\mu}({\zeta})\right)^{T}\Hat{C}_{22}, \label{6.11} 
\end{align}
where $ \Hat{C}_{22}=\left(\Hat{C}_{2}^{T}\mathcal{P}^{\mu}+d_{2}^{T}\right)^{T}$.\\
Again, approximating
\begin{align*}
\Hat{\Psi}^{\mu}({\zeta})\left(\Hat{\Psi}^{\mu}({\zeta})\right)^{T}\Hat{C}_{22}\approx\Hat{C}_{23}\Hat{\Psi}^{\mu}({\zeta}),
\end{align*}
where $\Hat{C}_{23}$ is the product operational matrix of order $\Hat{m}$ and given as
\begin{align*}
\Hat{C}_{23}=\Bigg(\int_{0}^{1} \left(\Hat{\Psi}^{\mu}({\zeta})\left(\Hat{\Psi}^{\mu}({\zeta})\right)^T\Hat{C}_{22}\left(\Hat{\Psi}^{\mu}({\zeta})\right)^{T}\right)d{\zeta}\Bigg)\Hat{D}^{-1}(\mu).
\end{align*}
Therefore, from Eq. \eqref{6.11} gives
\begin{align*}
x_{2}^2({\zeta}) &\approx\Hat{C}_{22}^{T}\Hat{C}_{23}\Hat{\Psi}^{\mu}({\zeta})\\
&=\Hat{C}_{24}^T\Hat{\Psi}^{\mu}({\zeta}),
\end{align*}
where $\Hat{C}_{24}^T =\Hat{C}_{22}^{T}\Hat{C}_{23}.$\\
Now,
\begin{align*}
b({\zeta})x_{2}^2({\zeta}) & \approx \left({B}^{T}\Hat{\Psi}^{\mu}({\zeta})\right)\left(\Hat{C}_{24}^{T}\Hat{\Psi}^{\mu}({\zeta})\right)^T\\
&={B}^{T}\Hat{\Psi}^{\mu}({\zeta})\left(\Hat{\Psi}^{\mu}({\zeta})\right)^T\Hat{C}_{24},
\end{align*}
The following approximation gives as
\begin{align*}
\Hat{\Psi}^{\mu}({\zeta})(\Hat{\Psi}^{\mu}({\zeta}))^{T}\Hat{C}_{24} \approx \Hat{C}_{25}\Hat{\Psi}^{\mu}({\zeta}),
\end{align*}
where $\Hat{C}_{25}$ is product operational matrix of order $\Hat{m}$ and given as
\begin{align*}
\Hat{C}_{25}=\Bigg(\int_{0}^{1} \left(\Hat{\Psi}^{\mu}({\zeta})\left(\Hat{\Psi}^{\mu}({\zeta})\right)^T\Hat{C}_{24}\left(\Hat{\Psi}^{\mu}({\zeta})\right)^{T}\right)d{\zeta}\Bigg)\Hat{D}^{-1}(\mu).
\end{align*}
Thus,
\begin{align*}
~~~~~~~~ b({\zeta})x_{21}^2({\zeta}) &  \approx B^{T}\Hat{C}_{25}\Hat{\Psi}^{\mu}({\zeta})\\
&=\Hat{C}_{26}^{T}\Hat{\Psi}^{\mu}({\zeta}),
\end{align*}
where $\Hat{C}_{26}^{T}= B^{T}\Hat{C}_{25}.$\\
Now, approximating $u^2({\zeta})$, yields
\begin{align}\nonumber
u^2({\zeta}) &\approx \left(\Hat{U}^{T} \Hat{\Psi}^{\mu}({\zeta})\right)^2 \\\nonumber
&= \Hat{U}^{T}\Hat{\Psi}^{\mu}({\zeta}) \left(\Hat{\Psi}^{\mu}({\zeta})\right)^{T} \Hat{U}\\  \nonumber
& \approx\Hat{ U}^{T}\Hat{ U}_{1}\Hat{\Psi}^{\mu}({\zeta})\\
&=\Hat{ U}_{2}^T\Hat{\Psi}^{\mu}({\zeta}),
\label{6.12}  
\end{align}
where $\Hat{ U}_{1}$ is product operational matrix of order $\Hat{m}$, given as
\begin{equation*}
    \Hat{ U}_{1}=\Big(\int_{0}^{1} \left(\Hat{\Psi}^{\mu}({\zeta})\left(\Hat{\Psi}^{\mu}({\zeta})\right)^T\Hat{U}\left(\Hat{\Psi}^{\mu}({\zeta})\right)^{T}\right)d{\zeta}\Big)\Hat{D}^{-1}(\mu),
\end{equation*}
and $$ \Hat{ U}_{2}^T=\Hat{U}^{T}\Hat{ U}_{1}.$$
Now, 
\begin{align}\nonumber
c({\zeta})u^2({\zeta}) &\approx \left({C}^{T} \Hat{\Psi}^{\mu}({\zeta})\right)\left(\Hat{ U}_{2}^{T}\Hat{\Psi}^{\mu}({\zeta})\right)^T \\ \nonumber
&= C^{T}\Hat{\Psi}^{\mu}({\zeta}) \left(\Hat{\Psi}^{\mu}({\zeta})\right)^{T} \Hat{ U}_{2}\\ \nonumber
& \approx C^{T} \Hat{ U}_{3}\Hat{\Psi}^{\mu}({\zeta})\\
&=\Hat{ U}_{4}^{T}\Hat{\Psi}^{\mu}({\zeta}),
 \label{6.13}
\end{align}
where $\Hat{ U}_{3}$ is product operational matrix of order $\Hat{m}$, given as
\begin{align*}
   \Hat{ U}_{3}=\Big(\int_{0}^{1} \left(\Hat{\Psi}^{\mu}({\zeta})\left(\Hat{\Psi}^{\mu}({\zeta})\right)^{T}\Hat{ U}_{2}\left(\Hat{\Psi}^{\mu}({\zeta})\right)^{T}\right)d{\zeta}\Big)\Hat{D}^{-1}(\mu),
\end{align*}
and $$ \Hat{ U}_{4}^{T}= C^{T}\Hat{ U}_{3}.$$
\\
Cost function has been approximated by using approximated value of functions and product operational matrices as
\begin{align}\nonumber
\Tilde{J}\approx & \dfrac{1}{2} \int_{0}^{1}\left(\Hat{C}_{16}^T+\Hat{C}_{26}^T+\Hat{U}_{4}^T\right)\Hat{\Psi}^{{\mu}}({\zeta})d{\zeta}\\
=&\Tilde{J}\left[\Hat{C}_{1}^T,\Hat{C}_{2}^T, \Hat{U}^T\right].\label{6.14}
\end{align}
The dynamical system can be approximated as 
\begin{align*}
\mathcal{G}_{1}\left(\Hat{C}_{1}^{T}\Hat{\Psi}^{{\mu}}({\zeta}),\Hat{C}_{2}^{T}\Hat{\Psi}^{{\mu}}({\zeta})\right)= \mathcal{H}_{1}\left(\Hat{C}_{12}^T\Hat{\Psi}^{{\mu}}({\zeta}), \Hat{C}_{22}^T\Hat{\Psi}^{{\mu}}({\zeta}),  \Hat{U}^T\Hat{\Psi}^{{\mu}}({\zeta})\right),\\
\left(\mathcal{G}_{1}\left(\Hat{C}_{1}^{T},\Hat{C}_{2}^{T}\right)- \mathcal{H}_{1}\left(\Hat{C}_{12}^T, \Hat{C}_{22}^T,  \Hat{U}^T\right)\right)\Hat{\Psi}^{{\mu}}({\zeta})=0.
\end{align*} 
\begin{align*}
\mathcal{G}_{2}\left(\Hat{C}_{1}^{T}\Hat{\Psi}^{{\mu}}({\zeta}),\Hat{C}_{2}^{T}\Hat{\Psi}^{{\mu}}({\zeta})\right)= \mathcal{H}_{2}\left(\Hat{C}_{12}^T\Hat{\Psi}^{{\mu}}({\zeta}), \Hat{C}_{22}^T\Hat{\Psi}^{{\mu}}({\zeta}),  \Hat{U}^T\Hat{\Psi}^{{\mu}}({\zeta})\right),\\
\left(\mathcal{G}_{2}\left(\Hat{C}_{1}^{T},\Hat{C}_{2}^{T}\right)- \mathcal{H}_{2}\left(\Hat{C}_{12}^T, \Hat{C}_{22}^T,  \Hat{U}^T\right)\right)\Hat{\Psi}^{{\mu}}({\zeta})=0. 
\end{align*} 
Thus dynamical system changes into system of algebraic equations:
\begin{align}\label{6.15}
\mathcal{G}_{1}\left(\Hat{C}_{1}^{T},\Hat{C}_{2}^{T}\right)- \mathcal{H}_{1}\left(\Hat{C}_{12}^T, \Hat{C}_{22}^T,  \Hat{U}^T\right)=0,\\ 
\mathcal{G}_{2}\left(\Hat{C}_{1}^{T},\Hat{C}_{2}^{T}\right)- \mathcal{H}_{2}\left(\Hat{C}_{12}^T, \Hat{C}_{22}^T,  \Hat{U}^T\right)=0.\label{6.16} 
\end{align}
Let
\begin{align}\nonumber
 \Tilde{J}^\star\left[\Hat{C}_{1},\Hat{C}_{2}, \Hat{U}, \eta^{\star},\lambda^{\star}\right] \approx\Tilde{J}\left[\Hat{C}_{1},\Hat{C}_{2}, \Hat{U}\right]+ \left( \mathcal{G}_{1}\left(\Hat{C}_{1}^{T},\Hat{C}_{2}^{T}\right)- \mathcal{H}_{1}\left(\Hat{C}_{12}^T, \Hat{C}_{22}^T,  \Hat{U}^T\right)\right)\eta^{\star}\\
 +\left(\mathcal{G}_{2}\left(\Hat{C}_{1}^{T},\Hat{C}_{2}^{T}\right)- \mathcal{H}_{2}\left(\Hat{C}_{12}^T, \Hat{C}_{22}^T,  \Hat{U}^T\right)\right)\lambda^{\star}, \label{6.17}   
\end{align}
where 
$\eta^{\star}=[ \eta_{1}^{\star}, \eta_{2}^{\star}, \dots,\eta_{\Hat{{m}}}^{\star}]^{T}$ and
$\lambda^{\star}=[ \lambda_{1}^{\star}, \lambda_{2}^{\star}, \dots,\lambda_{\Hat{{m}}}^{\star}]^{T}$
are the unknown Lagrange multipliers.\\\\
Now the necessary conditions for the extremum are
\begin{align}
 \dfrac{\partial{\Tilde{ J}^\star}} {\partial{\Hat{C}_{1}}}=0,\quad \dfrac{\partial{\Tilde{ J}^\star}} {\partial{\Hat{C}_{2}}}=0,\quad   \dfrac{\partial{\Tilde{J}^\star}} {\partial{\Hat{U}}}=0,\quad
  \dfrac{\partial{\Tilde{J}^\star}} {\partial{\eta^{\star}}}=0,\quad \dfrac{\partial{\Tilde{J}^\star}} {\partial{\lambda^{\star}}}=0.\label{6.18}
\end{align}
  The above system of linear  Eqs. \eqref{6.18} have been solved for $\Hat{C}_{1}, \Hat{C}_{2}, \Hat{U}, \eta^{\star}$ and $\lambda^{\star} $ using the Lagrange multipliers  method. By determining  $\Hat{C}_{1}, \Hat{C}_{2}$ and $\Hat{U}$, we have determined the approximated values of   $u({\zeta}), x_{1}({\zeta}),  $ and $x_{2}({\zeta})$ from Eqs. \eqref{6.7}, \eqref{6.8} and \eqref{6.9}, respectively. 
\section{Error estimation }
In this section, the best approximation to a smooth function has been obtained by estimating the error norm using FBWs.

\begin{lemma}\label{best approx}
Let $f(\zeta)\in C^N[a, b]$ and ${P}_{N-1}({\zeta})$ is the best interpolation polynomial to $f(\zeta)$ at the roots of the $N$-degree shifted Chebyshev polynomial in $[a, b]$. Then 
$$|f(\zeta)-{P}_{N-1}(\zeta)|\leq \dfrac{({b}-{a})^N}{N!2^{2N-1}}\max_{\xi\in[{a}, {b}]} |f^N(\xi)|.$$ 
\end{lemma}
\begin{proof}
For the proof of this lemma, one may refer to \cite{ray2018numerical}.
\end{proof}
\begin{lemma}\label{wavelets best approx}
Suppose $\displaystyle \sum_{{n}=1}^{2^{{k}-1}}\sum_{{m}=0}^{M-1}c_{{n}{m}}\psi_{{n}{m}}^{\mu}({\zeta})=\displaystyle \sum_{i=1}^{\Hat{m}}\Hat{c}_{i}\Hat{\psi}_{i}^{\mu}({\zeta})=\Hat{C}^{T}\Hat{\Psi}^{\mu}({\zeta})$  be the FBWs expansion of the real sufficiently smooth function ${f}({\zeta})\in [0,1]$. Then there exists a real number $\Tilde{\mathcal{M}}$ such that 
$$\lVert{f}({\zeta})-\Hat{C}^{T} \Hat{\Psi}^{\mu}({\zeta})\rVert_{2}\leq \dfrac{\Tilde{\mathcal{M}}}{\Hat{m}!2^{2\Hat{m}-1}}.$$
\end{lemma}
\begin{proof}
We are able to write,
$$\int_{0}^{1}\left({f}({\zeta})-\Hat{C}^{T} \Hat{\Psi}^{\mu}({\zeta})\right)^2d{\zeta}= \sum_{{n}=1}^{2^{{k}-1}M}\int_{\left(\frac{{n}-1}{2^{{k}-1}}\right)^{\frac{1}{\mu}}}^{\left(\frac{{n}}{2^{{k}-1}}\right)^{\frac{1}{\mu}}}\left({f}({\zeta})-\Hat{C}^{T} \Hat{\Psi}^{\mu}({\zeta})\right)^2d{\zeta},$$
On the subinterval $\left[\left(\frac{{n}-1}{2^{{k}-1}}\right)^{\frac{1}{\mu}},\left(\frac{{n}}{2^{{k}-1}}\right)^{\frac{1}{\mu}}\right)$, $\Hat{C}^{T}\Hat{\Psi}^{\mu}({\zeta})$
is a polynomial of degree at most $\Hat{m}= 2^{{k}-1}M$, that approximates ${f}$ with the least-square property.\\\\
As ${P}_{\Hat{m}-1}({\zeta})$ is the best approximation to ${f}({\zeta})$, that agrees with ${f}({\zeta})$ at the zeros of shifted Chebyshev polynomial. Therefore, the following has been obtained by using Lemma \ref{best approx}:
\begin{equation*}
\begin{split}
\lVert{f}({\zeta})-\Hat{C}^{T} \Hat{\Psi}^{\mu}({\zeta})\rVert_{2}&=\sum_{{n}=1}^{2^{{k}-1}M}\int_{\left(\frac{{n}-1}{2^{{k}-1}}\right)^{\frac{1}{\mu}}}^{\left(\frac{{n}}{2^{{k}-1}}\right)^{\frac{1}{\mu}}}\left({f}({\zeta})-\Hat{C}^{T} \Hat{\Psi}^{\mu}({\zeta})\right)^2d{\zeta} \\
& \leq\sum_{{n}=1}^{2^{{k}-1}M}\int_{\left(\frac{{n}-1}{2^{{k}-1}}\right)^{\frac{1}{\mu}}}^{\left(\frac{{n}}{2^{{k}-1}}\right)^{\frac{1}{\mu}}}\left({f}({\zeta})-{P}_{\Hat{m}-1}({\zeta})\right)^2d{\zeta} \\
&\leq \int_{0}^{1}\Big(\dfrac{1}{\Hat{m}!2^{2\Hat{m}-1}}\max_{\xi\in [0, 1]}|{f}^{(\Hat{m})}(\xi)|\Big)^2 d{\zeta}.
\end{split}   
\end{equation*}
Let us consider that there exists a real number $\Tilde{\mathcal{M}}$ such that
$$\max_{\xi\in[0, 1]} |{f}^{(\Hat{m})}(\xi)| \leq \Tilde{\mathcal{M}}.$$
Thus,
\begin{align*}
\begin{split}
\lVert{f}({\zeta})-\Hat{C}^{T}\Hat{\Psi}^{\mu}({\zeta})\rVert_{2}^{2} & = \int_{0}^{1}\left({f}({\zeta})-\Hat{C}^{T} \Hat{\Psi}^{\mu}({\zeta})\right)^2d{\zeta}  \\
&\leq \int_{0}^{1}\Big(\dfrac{\Tilde{\mathcal{M}}}{\Hat{m}!2^{2\Hat{m}-1}}\Big)^2 d{\zeta}.\\
\end{split}    
\end{align*}
This implies that
\begin{align*}
\lVert{f}({\zeta})-\Hat{C}^{T} \Hat{\Psi}^{\mu}({\zeta})\rVert_{2} \leq \dfrac{\Tilde{\mathcal{M}}}{\Hat{m}!2^{2\Hat{m}-1}}.
\end{align*}
\end{proof}
Now an error estimate concerning  $\displaystyle |\inf_{\Gamma_{\Hat{m}}} \Tilde{J}-\inf_{\Gamma}\Tilde{J}|$  has been established   for the proposed method.
By considering $z_{1}({\zeta})= {^{C}}\mathcal{D}_{{\zeta}}^{\mu}x_{1}({\zeta})$ and $z_{2}({\zeta})= {^{C}}\mathcal{D}_{{\zeta}}^{\mu}x_{2}({\zeta})$ the problem Eqs. \eqref{1.1}, \eqref{1.2} and \eqref{1.3} are equivalent to the following problem:
\begin{equation}
\begin{split}
\min\Tilde{J} & = \int_{0}^{1} \mathcal{F}\left({\zeta}, \mathcal{I}^{\mu}(z_{1}({\zeta}))+x_{1}(0), \mathcal{I}^{\mu}(z_{2}({\zeta}))+x_{2}(0), u({\zeta})\right) d{\zeta}\\
&= \int_{0}^{1} \mathcal{F}({\zeta}, z_{1}({\zeta}),z_{2}({\zeta}), u({\zeta})) d{\zeta},
\end{split}\label{7.1}
\end{equation}
\begin{align}  
\mathcal{G}_{1}\left({^{C}}\mathcal{D}_{{\zeta}}^{{\mu}}z_{1}({\zeta}),{^{C}}\mathcal{D}_{{\zeta}}^{{\mu}}z_{2}({\zeta}) \right)= \mathcal{H}_{1}({\zeta}, z_{1}({\zeta}), z_{2}({\zeta}), u({\zeta})),\label{7.2}
\end{align}
\begin{align}  
\mathcal{G}_{2}\left({^{C}}\mathcal{D}_{{\zeta}}^{{\mu}}z_{1}({\zeta}),{^{C}}\mathcal{D}_{{\zeta}}^{{\mu}}z_{2}({\zeta}) \right)= \mathcal{H}_{2}({\zeta}, z_{1}({\zeta}), z_{2}({\zeta}), u({\zeta})).\label{7.3}
\end{align}
\\
\begin{theorem}
 The set $\Gamma$ is consisting of all Lipschitz functions $(z_{1}({\zeta}), z_{2}({\zeta}),  u({\zeta}))$ that satisfy Eqs. \eqref{7.2} and \eqref{7.3} and $\Gamma_{\Hat{m}}$ is a subset of $\Gamma$ consisting of all functions   $\displaystyle\left(\sum_{k=1}^{\Hat{m}}z_{1k}^{\star}\Psi_{k}^{{\mu}}({\zeta}), \sum_{k=1}^{\Hat{m}}z_{2k}^{\star}\Psi_{k}^{{\mu}}({\zeta}),   \\
 \sum_{k=1}^{\Hat{m}}u_{k}^{\star}\Psi_{k}^{{\mu}}({\zeta})\right)$. Then there exists real
numbers $\Tilde{\mathcal{M}}^{\star}$ and $\Tilde{\mathcal{L}} >0$, such that
\begin{align*}
\displaystyle|\inf_{\Gamma_{\Hat{m}}}\Tilde{J} - \displaystyle\inf_{\Gamma}\Tilde{J}|\leq \dfrac{\Tilde{\mathcal{L}}\Tilde{\mathcal{M}}^{\star}}{\Hat{m}!2^{2\Hat{m}-1}}.
\end{align*}
\end{theorem}
\begin{proof}
Assume that $\mathcal{F}$ is Lipschitz  continuous function with respect to $(z_{1}({\zeta}), z_{2}({\zeta}), u({\zeta}))$ with a Lipschitz constant $\Tilde{\mathcal{L}}$.
\begin{align}\nonumber
|\mathcal{F}({\zeta}, z_{1}, z_{2}, u)-\mathcal{F}({\zeta},   \Tilde{z}_{1}, \Tilde{z}_{2},  \Tilde{u})|\leq &\Tilde{\mathcal{L}}\left(||z_{1}-\Tilde{z}_{1}||_{2}+||z_{2}-\Tilde{z}_{2}||_{2}+||u-\Tilde{u}||_{2}\right),\label{7.4}
 \end{align}   
for all ${\zeta}\in [0,1]$, for all $\Tilde{z}_{1}, \Tilde{z}_{2},  \Tilde{u}, z_{1}, z_{2}, u $  belong to $L^2[0,1]$. Since $\Gamma_{\Hat{m}} \subseteq \Gamma $,
 therefore $\displaystyle\inf_{\Gamma_{\Hat{m}}}\Tilde{J}\geq $ 
 $\displaystyle\inf_{\Gamma} \Tilde{J} $.
 \\
Let $(z_{1}({\zeta}), z_{2}({\zeta}), u({\zeta}))$ be the exact result for state functions and control function
and
$\displaystyle \Big(\sum_{k=1}^{\Hat{m}}z_{1k}^{\star}\Psi_{k}^{{\mu}}({\zeta}),\\ \sum_{k=1}^{\Hat{m}}z_{2k}^{\star}\Psi_{k}^{{\mu}}({\zeta}), \sum_{k=1}^{\Hat{m}}u_{k}^{\star}\Psi_{k}^{{\mu}}({\zeta})\Big)$ be the FBWs expansion of $(z_{1}({\zeta}), z_{2}({\zeta}), u({\zeta}))$, then

\begin{align*}
\begin{split}
|\displaystyle\inf_{\Gamma_{\Hat{m}}}\Tilde{J} - \displaystyle& \inf_{\Gamma}\Tilde{J}| \leq \Tilde{J}\displaystyle \left(\sum_{k=1}^{\Hat{m}}z_{1k}^{\star}\Psi_{k}^{{\mu}}({\zeta}), \sum_{k=1}^{\Hat{m}}z_{2k}^{\star}\Psi_{k}^{{\mu}}({\zeta}),  \sum_{k=1}^{\Hat{m}}u_{k}^{\star}\Psi_{k}^{{\mu}}({\zeta})\right)-\Tilde{J}(z_{1}({\zeta}), z_{2}({\zeta}),  u({\zeta}))\\  
= & \left|\int_{0}^{1}\mathcal{F}\displaystyle \left({\zeta}, \sum_{k=1}^{\Hat{m}}z_{1k}^{\star}\Psi_{k}^{{\mu}}({\zeta}), \sum_{k=1}^{\Hat{m}}z_{2k}^{\star}\Psi_{k}^{{\mu}}({\zeta}),  \sum_{k=1}^{\Hat{m}}u_{k}^{\star}\Psi_{k}^{{\mu}}({\zeta})\right)d{\zeta}
-\int_{0}^{1} \mathcal{F}({\zeta}, (z_{1}({\zeta}), z_{2}({\zeta}), u({\zeta}))d{\zeta}\right| \\
\end{split}
\end{align*}
\begin{align*}
\begin{split}
\leq &\int_{0}^{1}\left|\mathcal{F}\left({\zeta}, \sum_{k=1}^{\Hat{m}}z_{1k}^{\star}\Psi_{k}^{{\mu}}({\zeta}), z_{2}({\zeta}), u({\zeta}))\right)- \mathcal{F}({\zeta}, (z_{1}({\zeta}), z_{2}({\zeta}),  u({\zeta}))\right|d{\zeta}\\
& + \int_{0}^{1}\left|\mathcal{F}\left({\zeta}, \sum_{k=1}^{\Hat{m}}z_{1k}^{\star}\Psi_{k}^{{\mu}}({\zeta}), \sum_{k=1}^{\Hat{m}}z_{2k}^{\star}\Psi_{k}^{{\mu}}({\zeta}), u({\zeta}))\right) - \mathcal{F}\left({\zeta}, \sum_{k=1}^{\Hat{m}}z_{1k}^{\star}\Psi_{k}^{{\mu}}({\zeta}), z_{2}({\zeta}), u({\zeta}))\right)\right|d{\zeta} \\
&+\int_{0}^{1} \left|\mathcal{F}\displaystyle \left({\zeta}, \sum_{k=1}^{\Hat{m}}z_{1k}^{\star}\Psi_{k}^{{\mu}}({\zeta}), \sum_{k=1}^{\Hat{m}}z_{2k}^{\star}\Psi_{k}^{{\mu}}({\zeta}),  \sum_{k=1}^{\Hat{m}}u_{k}^{\star}\Psi_{k}^{{\mu}}({\zeta})\right)\right.\\ 
&-\left.\mathcal{F}\displaystyle \left({\zeta}, \sum_{k=1}^{\Hat{m}}z_{1k}^{\star}\Psi_{k}^{{\mu}}({\zeta}), \sum_{k=1}^{\Hat{m}}z_{2k}^{\star}\Psi_{k}^{{\mu}}({\zeta}), u({\zeta})\right) \right| d{\zeta},
\end{split}
\end{align*}
by using Lemma (7.2)
\begin{align*}
\begin{split}
|\displaystyle\inf_{\Gamma_{\Hat{m}}}\Tilde{J} - \displaystyle\inf_{\Gamma}\Tilde{J}| & \leq \int_{0}^{1} \Tilde{\mathcal{L}} \left(||z_{1} -  \sum_{k=1}^{\Hat{m}}z_{1k}^{\star}\Psi_{k}^{{\mu}}({\zeta})||_{2}+||z _{2} - \sum_{k=1}^{\Hat{m}}z_{2k}^{\star}\Psi_{k}^{{\mu}}({\zeta})||_{2}+||u-\sum_{k=1}^{\Hat{m}}u_{k}^{\star}\Psi_{k}^{{\mu}}(\zeta)|_{2}\right)d{\zeta}\\
 & \leq \Tilde{\mathcal{L}} \left(\dfrac{\Tilde{\mathcal{M}}_{1}}{\Hat{m}!2^{2\Hat{m}-1}}+\dfrac{\Tilde{\mathcal{M}}_{2}}{\Hat{m}!2^{2\Hat{m}-1}}+\dfrac{\Tilde{\mathcal{M}}_{3}}{\Hat{m}!2^{2\Hat{m}-1}}\right)\\
&=  \Tilde{\mathcal{L}} \left(\dfrac{\Tilde{\mathcal{M}}_{1}+\Tilde{\mathcal{M}}_{2}+\Tilde{\mathcal{M}}_{3}}{\Hat{m}!2^{2\Hat{m}-1}}\right),
\end{split}  
\end{align*}
where $\Tilde{\mathcal{M}}_{1}=\displaystyle\max_{{\zeta}\in[0,1]}|z_{1}^{\Hat{m}}({\zeta})|$, $\Tilde{\mathcal{M}}_{2}=\displaystyle\max_{{\zeta}\in[0,1]}|z_{2}^{\Hat{m}}({\zeta})|$
and  $ \Tilde{\mathcal{M}}_{3}=\displaystyle\max_{{\zeta}\in[0,1]}|u^{\Hat{m}}({\zeta})|$.
\begin{align*}
|\displaystyle\inf_{\Gamma_{\Hat{m}}}\Tilde{J} - \displaystyle\inf_{\Gamma}\Tilde{J}|\leq \dfrac{\Tilde{\mathcal{M}}^{\star}}{\Hat{m}!2^{2\Hat{m}-1}}, 
\end{align*}
where  $\Tilde{\mathcal{M}}^{\star}= \Tilde{\mathcal{M}}_{1}+ \Tilde{\mathcal{M}}_{2}+\Tilde{\mathcal{M}}_{3}$.\\\\
\end{proof}
\section{ Convergence analysis}
In this section, convergence analysis has been analyzed for multi-dimensional FOCP.
\begin{theorem}
The approximate results $z_{1}({\zeta})\approx\Hat{Z}_{1}^T\Hat{\Psi}^{{\mu}}({\zeta}), z_{2}({\zeta})\approx\Hat{Z}_{2}^T\Hat{\Psi}^{{\mu}}({\zeta})$ and  $u({\zeta})\approx\Hat{U}^T\Hat{\Psi}^{{\mu}}({\zeta})$, converge respectively to the exact result as $\Hat{m}$, the number of the FBWs tends to $\infty$.\\
\end{theorem}
\begin{proof} 
 Suppose $\Gamma _{\Hat{m}}$ is the set of all $\left(\Hat{Z}_{1}^T\Hat{\Psi}^{{\mu}}({\zeta}),\Hat{Z}_{2}^T\Hat{\Psi}^{{\mu}}({\zeta}), \Hat{U}^T\Hat{\Psi}^{{\mu}}({\zeta})\right) $ which satisfies the constraint given in Eqs. \eqref{7.2} and \eqref{7.3}.\\ 
Using convergence property of FBWs, for each 
$\left(\Hat{Z}_{11}^T\Hat{\Psi}^{{\mu}}({\zeta}),\Hat{Z}_{12}^T\Hat{\Psi}^{{\mu}}({\zeta}), \Hat{U}_{1}^T\Hat{\Psi}^{{\mu}}({\zeta})\right)
\in \Gamma_{\Hat{m}} $, there exists a unique pair of functions
 $(z_{11}({\zeta}), z_{12}({\zeta}), u_{1}({\zeta})) $  such that
 $$\left(\Hat{Z}_{11}^T\Hat{\Psi}^{{\mu}}({\zeta}),\Hat{Z}_{12}^T\Hat{\Psi}^{{\mu}}({\zeta}), \Hat{U}_{1}^T\Hat{\Psi}^{{\mu}}({\zeta})\right)\to (z_{11}({\zeta}), z_{12}({\zeta}),  u_{1}({\zeta})),$$ 
 as $ \Hat{m}\to \infty $.
 It is obvious that $(z_{11}({\zeta}), z_{12}({\zeta}),  u_{1}({\zeta}))\in \Gamma,$ where $\Gamma$ is the set of all results that satisfy the constraint given in Eq. \eqref{7.2}. So as $\Hat{m}\to \infty,$ each element in $\Gamma_{\Hat{m}}$ tends to an element in $\Gamma$.\\\\
 Furthermore,  as $\Hat{m}\to \infty$, then \\
   $$\Tilde{J}_{1}^{\Hat{m}} = \Tilde{J}\left(\Hat{Z}_{11}^T\Hat{\Psi}^{{\mu}}({\zeta}),\Hat{Z}_{12}^T\Hat{\Psi}^{{\mu}}({\zeta}), \Hat{U}_{1}^T\Hat{\Psi}^{{\mu}}({\zeta})\right)\to \Tilde{J}_{1},$$
 where $\Tilde{J}_{1}^{\Hat{m}}$ is the value of the cost
function corresponding to the pair $\left(\Hat{Z}_{11}^T\Hat{\Psi}^{{\mu}}({\zeta}),\Hat{Z}_{12}^T\Hat{\Psi}^{{\mu}}({\zeta}), \Hat{U}_{1}^T\Hat{\Psi}^{{\mu}}({\zeta})\right)$ and $\Tilde{J}_{1}$ is the objective value corresponding to the feasible result  $(z_{11}({\zeta}), z_{12}({\zeta}), u_{1}({\zeta}))$.
\begin{align*}
\Gamma_{1} \subseteq \Gamma_{2} \subseteq \dots \subseteq \Gamma_{\Hat{m}}\subseteq \Gamma_{\Hat{m}+1}\subseteq \dots\subseteq \Gamma,
\end{align*} 
then
\begin{align*}
\inf_{\Gamma_{1}}{\Tilde{J}_{1}} \geq \inf_{\Gamma_{2}}{\Tilde{J}_{2}} \geq \dots \geq \inf_{\Gamma_{\Hat{m}}}{\Tilde{J}_{\Hat{m}}} \geq \inf_{\Gamma_{\Hat{m}+1}}{\Tilde{J}_{\Hat{m}+1}}\geq \dots\geq\inf_{\Gamma}{\Tilde{J}},
\end{align*}
which is a decreasing and bounded sequence. Since every bounded monotone sequence is convergent.
Therefore, it converges to a number $\omega$, such that
$$ \omega \geq \inf_{\Gamma}\Tilde{J}.$$\\
Next, we have to show that 
$$ \omega=\lim_{\Hat{m}\to \infty}\inf_{\Gamma_{\Hat{m}}}{\Tilde{J}_{\Hat{m}}} =\inf_{\Gamma}\Tilde{\Tilde{J}}. $$
Given $ \epsilon >0 $, let $\left(z_{1}({\zeta}), z_{2}({\zeta}), u({\zeta})\right)\in \Gamma $ such that 
\begin{equation}
  \Tilde{J}(z_{1}({\zeta}), z_{2}({\zeta}),  u({\zeta})) < \inf_{\Gamma}\Tilde{J}+\epsilon.\label{8.1}
\end{equation}
Since $\Tilde{J}(z_{1}({\zeta}), z_{2}({\zeta}),  u_{1}({\zeta}))$ is continuous, for this value
of $\epsilon$, there exists $K(\epsilon)$ so that if $\Hat{m}>K(\epsilon), $ then\\
\begin{align*}
\left|\Tilde{J}(z_{1}({\zeta}), z_{2}({\zeta}),  u({\zeta}))-\Tilde{J}\left(\Hat{Z}_{1}^T\Hat{\Psi}^{{\mu}}({\zeta}),\Hat{Z}_{2}^T\Hat{\Psi}^{{\mu}}({\zeta}), \Hat{U}^T\Hat{\Psi}^{{\mu}}({\zeta})\right)\right|< \epsilon.
\end{align*} 
This implies that
\begin{equation}
\Tilde{J}\left(\Hat{Z}_{1}^T\Hat{\Psi}^{{\mu}}({\zeta}),\Hat{Z}_{2}^T\Hat{\Psi}^{{\mu}}({\zeta}), \Hat{U}^T\Hat{\Psi}^{{\mu}}({\zeta})\right) < \Tilde{J}(z_{1}({\zeta}), z_{2}({\zeta}),  u({\zeta}))+\epsilon.
\label{8.2} 
\end{equation}
Using Eqs. \eqref{8.1} and \eqref{8.2}, it follows that
\begin{align}
\Tilde{J}\left(\Hat{Z}_{1}^T\Hat{\Psi}^{{\mu}}({\zeta}),\Hat{Z}_{2}^T\Hat{\Psi}^{{\mu}}({\zeta}),\Hat{U}^T\Hat{\Psi}^{{\mu}}({\zeta})\right)   < \Tilde{J}(z_{1}({\zeta}), z_{2}({\zeta}),  u({\zeta}))+\epsilon<\inf_{\Gamma}\Tilde{J}+2\epsilon.  \label{8.3} 
\end{align}
\\
Now
\begin{align}
\inf_{\Gamma}\Tilde{J}\leq  \inf_{\Gamma_{\Hat{m}}}\Tilde{J}_{\Hat{m}}\leq \Tilde{J}\left(\Hat{Z}_{1}^T\Hat{\Psi}^{{\mu}}({\zeta}),\Hat{Z}_{2}^T\Hat{\Psi}^{{\mu}}({\zeta}), \Hat{U}_{1}^T\Hat{\Psi}^{{\mu}}({\zeta})\right).\label{8.4}
\end{align}
From Eqs. \eqref{8.3} and \eqref{8.4}, one get
\begin{align*}
   \inf_{\Gamma}\Tilde{J}\leq  \inf_{\Gamma_{\Hat{m}}}\Tilde{J}_{\Hat{m}}<\inf_{\Gamma}\Tilde{J}+2\epsilon.  
\end{align*}
Whence
\begin{align*}
  0\leq \inf_{\Gamma_{\Hat{m}}}\Tilde{J}_{\Hat{m}}-\inf_{\Gamma}\Tilde{J}< 2\epsilon, 
\end{align*}
where $\epsilon$ is chosen arbitrary.\\\\
Hence
\begin{align*}
\omega=\lim_{\Hat{m}\to \infty}\inf_{\Gamma_{\Hat{m}}}{\Tilde{J}_{\Hat{m}}} =\inf_{\Gamma}\Tilde{J}.
\end{align*}
\end{proof}

\section{Numerical Problems}
In this section, the accuracy and utility of the proposed method are demonstrated with some numerical examples.
\begin{problem}
Let the following two-dimensional fractional linear quadratic time-variant problem \textnormal{\cite{dehestani2022numerical}} 
\begin{align*}\nonumber
\min\Tilde{J}=\dfrac{1}{2}\int_{0}^{1}\left(x_{1}^2({\zeta})+x_{2}^2({\zeta})+u^2({\zeta})\right)d{\zeta},
\end{align*}
$$
^{C}\mathcal{D}_{{\zeta}}^{{\mu}}x_{1}({\zeta})=-x_{1}({\zeta})+x_{2}({\zeta})+u({\zeta}),
$$
$$^{C}\mathcal{D}_{{\zeta}}^{{\mu}}x_{2}({\zeta})=-2x_{2}({\zeta}),$$
$$x_{1}(0)=1, x_{2}(0)=1.$$   
\end{problem}  
Our aim is to find  state functions $x_{1}({\zeta})$, $ x_{2}({\zeta})$ and control function $u({\zeta})$  which minimize the cost function $\Tilde{J}$.  The exact result of the above problem for
${\mu} = 1$ as follows:
$$ x_{1}({\zeta})=-\dfrac{3}{2}e^{-2{\zeta}}+2.48164 e^{-\sqrt{2}
{\zeta}}+0.018352 e^{\sqrt{2}{\zeta}}, $$
$$ x_{2}({\zeta})=e^{-2{\zeta}}, $$
$$ u({\zeta})=\dfrac{1}{2}e^{-2{\zeta}}-1.02793e^{-\sqrt{2}{\zeta}}+0.0443056e^{\sqrt{2}{\zeta}}. $$ 
This multi-dimensional FOCP has been solved by the proposed methods using OBWs and FBWs. Table 1 shows the approximate values of the cost function for OBWs and FBWs, and it is also observed that FBWs give a better result than the OBWs. The approximate results obtained by OBWs and FBWs methods for different values of $\mu$ for state functions $x_{1}(\zeta)$, $ x_{2}(\zeta)$ and control function $u(\zeta)$ are shown in Tables 2, 3, and 4, respectively. Table 5 shows the absolute error of state and control functions when $\mu =1$ and at different values of $k $ and $M$. Fig. 1 shows the exact result when $\mu=1$ and approximate results when $\mu=0.8,0.9$, and $0.99$ of the state functions $x_{1}(\zeta), x_{2}(\zeta)$ and control function $u(\zeta)$ for the OBWs method. Fig. 2 shows the exact result when $\mu=1$ and approximate results when $\mu=0.8,0.9$, and $0.99$ of the state functions $x_{1}(\zeta), x_{2}(\zeta)$ and control function $u(\zeta)$ for FBWs method. Fig. 3 shows the exact and approximate results of the state functions $x_{1}(\zeta), x_{2}(\zeta)$ and control function $u(\zeta)$ when $k=2, M=3$ and $\mu=1$ using FBWs. Fig. 4 shows the absolute error results of state and control functions using FBWs. As can be seen in Table 5, the approximate results of the state and control functions converge to the exact result as the number of fractional Bernoulli wavelets bases increases.\\ 
\begin{figure}
\centering
\begin{subfigure}{0.45\textwidth}
    \includegraphics[width=\textwidth]{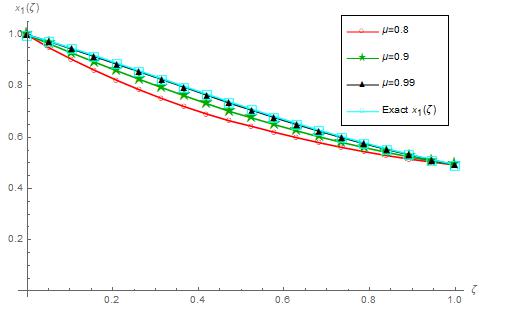}
    \caption{Approximate and exact results of $x_{1}(\zeta)$.}
    \label{1(a)p2}
\end{subfigure}
\hfill
\begin{subfigure}{0.45\textwidth}
    \includegraphics[width=\textwidth]{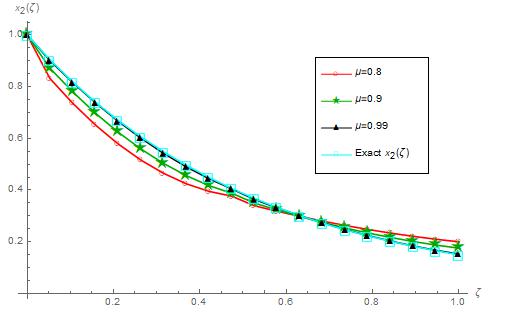}
    \caption{Approximate and exact results of  $x_{2}(\zeta)$.}
    \label{1(b)p2}
\end{subfigure}
\hfill
\begin{subfigure}{0.45\textwidth}
    \includegraphics[width=\textwidth]{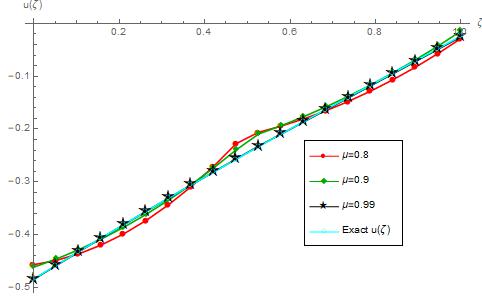}
    \caption{Approximate and exact results of $u(\zeta)$.}
    \label{1(c)p2}
\end{subfigure}
        
\caption{OBWs results when $k=2$ and $M=3$.}
\label{1p2}
\end{figure}
\begin{figure}
\centering
\begin{subfigure}{0.4\textwidth}
    \includegraphics[width=\textwidth]{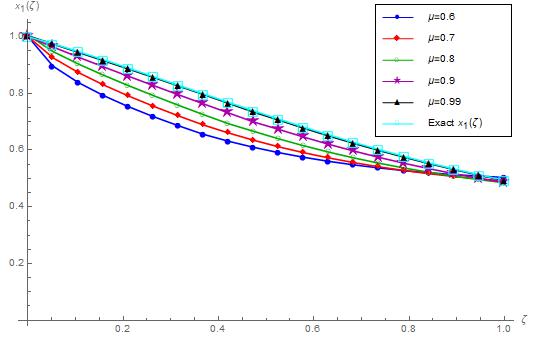}
    \caption{Approximate and exact result of $x_{1}(\zeta)$.}
    \label{2(a)p2}
\end{subfigure}
\hfill
\begin{subfigure}{0.4\textwidth}
    \includegraphics[width=\textwidth]{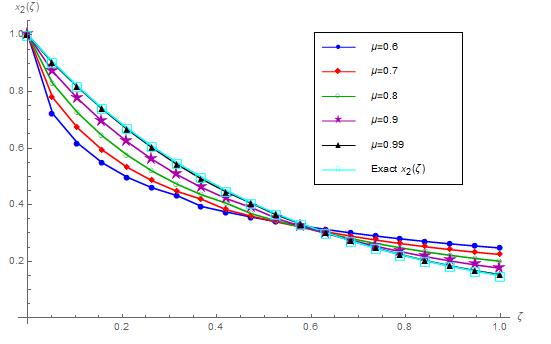}
    \caption{Approximate and exact result of $x_{1}(\zeta)$.}
    \label{2(b)p2}
\end{subfigure}
\hfill
\begin{subfigure}{0.4\textwidth}
    \includegraphics[width=\textwidth]{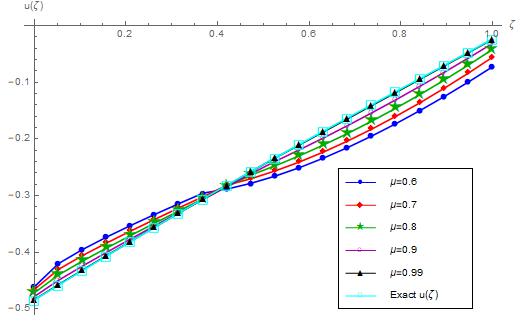}
    \caption{Approximate and exact result of $u(\zeta)$.}
    \label{2(c)p2}
\end{subfigure}
        
\caption{FBWs results when $k=2$ and $ M=3$.}
\label{2p2}
\end{figure}

\begin{figure}
\centering
\begin{subfigure}{0.4\textwidth}
    \includegraphics[width=\textwidth]{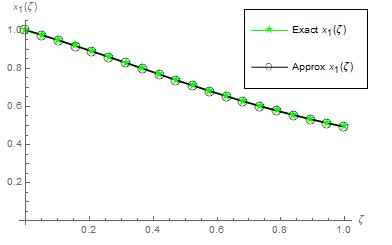}
    \caption{Exact and approximate results of $x_{1}(\zeta)$.}
    \label{3(a)p2}
\end{subfigure}
\hfill
\begin{subfigure}{0.4\textwidth}
    \includegraphics[width=\textwidth]{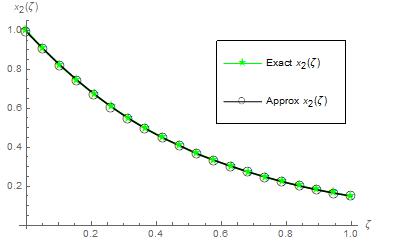}
    \caption{Exact and approximate results of $x_{1}(\zeta)$.}
    \label{3(b)p2}
\end{subfigure}
\hfill
\begin{subfigure}{0.4\textwidth}
    \includegraphics[width=\textwidth]{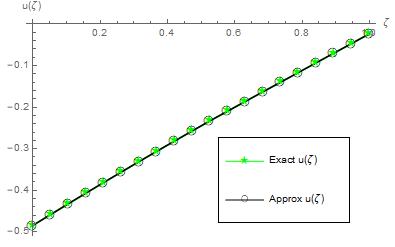}
    \caption{Exact and approximate results of $u(\zeta)$.}
    \label{3(c)p2}
\end{subfigure}
        
\caption{Exact and approximate results of  the state and control functions when $k = 2$, $ M = 3$ and $\mu = 1$ using FBWs.}
\label{3p2}
\end{figure}

\begin{figure}
\centering
\begin{subfigure}{0.4\textwidth}
    \includegraphics[width=\textwidth]{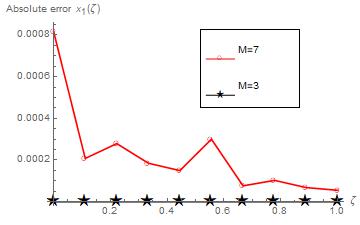}
    \caption{Absolute error of $x_{1}(\zeta)$ when $k=3$.}
    \label{4(a)p2}
\end{subfigure}
\hfill
\begin{subfigure}{0.4\textwidth}
    \includegraphics[width=\textwidth]{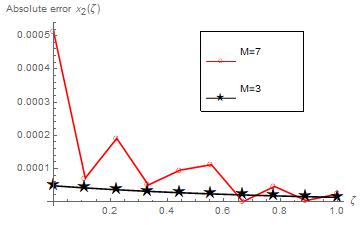}
    \caption{Absolute error of $x_{2}(\zeta)$ when $k=3$.}
    \label{4(b)p2}
\end{subfigure}
\hfill
\begin{subfigure}{0.4\textwidth}
    \includegraphics[width=\textwidth]{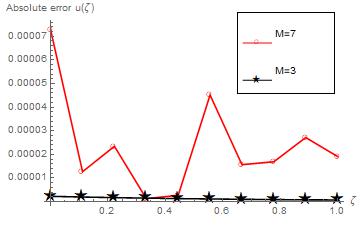}
    \caption{Absolute error of $u(\zeta) $ when $k=3$.}
    \label{4(c)p2}
\end{subfigure}
        
\caption{Absolute error of state and control functions when $k=3, M=3$ and $k=3, M=7$ using FBWs.}
\label{4p2}
\end{figure}

\begin{table}[h]
\centering
\caption{Approximate values of the cost function $\Tilde{{J}}$ for different values of  $\mu$.}
\begin{tabular}{|c|c|c|c|}\hline
$\mu $& OBWs    & FBWs \\ \cline{1-3}
\hline
1 & 0.431987 & 0.431987\\
\hline
0.99 & 0.429029 & 0.429028\\
\hline
0.9 & 0.403542 &  0.403422 \\
\hline
0.8 & 0.377656 &0.377388\\
\hline
0.7 & 0.355052 & 0.354417\\
\hline
0.6& 0.336731&0.335404 \\
\hline
0.5 & 0.324598 &0.322129\\
\hline
\end{tabular}
\end{table}

\begin{landscape}
\begin{table}
\centering
\small\addtolength{\tabcolsep}{-1pt}
\caption{Approximate results of the state function $x_{1}({\zeta})$ for $\mu=0.6, 0.7, 0.8, 0.9$ and $ 0.99.$ }
\begin{tabular}{|c|c|c|c|c|c|c|c|c|c|c|c|c|c|c|} \hline
${\zeta}$ &  \multicolumn{2}{|c|}{$\mu=0.6$} &\multicolumn{2}{|c|}{$\mu=0.7$} &\multicolumn{2}{|c|}{$\mu=0.8$}& \multicolumn{2}{|c|}{$\mu=0.9$} &\multicolumn{2}{|c|}{$\mu=0.99$}\\[0.5ex] \cline{2-11}
& OBWs &  FBWs & OBWs &  FBWs &  OBWs &  FBWs &  OBWs &  FBWs &  OBWs &  FBWs \\ \cline{2-11}
\hline
0.1 &0.839425& 0.838387& 0.873411& 0.87381& 0.903723& 0.903798& 0.928614&0.927822& 0.944683&0.944739\\
\hline
0.2&0.73625& 0.753203& 0.779968& 0.790936& 0.822034& 0.827442& 0.859362& 0.860354& 0.88579& 0.886\\
\hline
0.3 &0.663489& 0.685737& 0.706357& 0.720562& 0.750938& 0.758098& 
0.793462& 0.79508& 0.825812&0.826086\\
 \hline
0.4&0.621142& 0.630914& 0.652577& 0.660207& 0.690435& 0.692795& 
  0.730915& 0.730828& 0.764746& 0.764905\\
\hline
0.5&0.595014& 0.590831&0.615623& 0.611349& 0.642464& 0.63942& 0.67549&0.672819& 0.705649& 0.705731\\
\hline
0.6&0.568937&0.560139& 0.579379&0.572119& 0.598194& 0.592673& 
  0.624705& 0.620187& 0.649701& 0.649667\\
\hline
0.7&0.547107& 0.536896& 0.54905& 0.540895& 0.560197& 0.553602& 0.579739&0.5741& 0.598626& 0.598554\\
\hline
0.8&0.529526& 0.519791& 0.524635& 0.516546& 0.528474& 0.521422& 
0.540592& 0.534172& 0.552424&0.552359\\
\hline
0.9&0.516192& 0.507887& 0.506135& 0.498239&0.503024& 0.495532&
  0.507264& 0.500104& 0.511095& 0.511052\\
\hline
 
\end{tabular}
\end{table}

\begin{table}
\centering
\small\addtolength{\tabcolsep}{-1pt}
\caption{Approximate results of the state function $x_{2}({\zeta})$ for $\mu=0.6, 0.7, 0.8, 0.9$ and $ 0.99.$ }
\begin{tabular}{|c|c|c|c|c|c|c|c|c|c|c|c|c|c|c|} \hline
${\zeta}$ &  \multicolumn{2}{|c|}{$\mu=0.6$} &\multicolumn{2}{|c|}{$\mu=0.7$} &\multicolumn{2}{|c|}{$\mu=0.8$}& \multicolumn{2}{|c|}{$\mu=0.9$} &\multicolumn{2}{|c|}{$\mu=0.99$}\\[0.5ex] \cline{2-11}
& OBWs &  FBWs & OBWs &  FBWs &  OBWs &  FBWs &  OBWs &  FBWs &  OBWs &  FBWs \\ \cline{2-11}
\hline
0.1 &0.643398& 0.617907& 0.69029& 0.671598& 0.736724& 0.725343& 0.780671& 0.775766& 0.816923&0.816512\\
\hline
0.2&0.502437& 0.497323& 0.537628& 0.532373& 0.579033& 0.575139& 
0.624441& 0.62253& 0.666788& 0.666614\\
\hline
0.3 &0.412169& 0.432434& 0.434285& 0.447517& 0.46519& 0.472632& 0.503451&0.506498& 0.542598& 0.542848\\
 \hline
0.4&0.372592& 0.374& 0.380263& 0.381825& 0.395196& 0.404867& 0.417699&0.42153& 0.444353& 0.444655\\
\hline
0.5&0.327607& 0.34033& 0.332649&, 0.3394& 0.339513& 0.343085& 0.349921& 0.351605& 0.364058& 0.364235\\
\hline 
0.6&0.306212& 0.31257& 0.301578& 0.304307& 0.298493& 0.299469& 0.298021&0.298306& 0.301282& 0.301305\\
\hline
0.7&0.286626& 0.289477& 0.274366& 0.27529& 0.2632& 0.263246& 0.253624&
  0.253479& 0.247501& 0.24748\\
  \hline
0.8&0.26885& 0.270216& 0.251014& 0.251477& 0.233633& 0.233657& 0.21673& 
  0.216651& 0.202716& 0.202705\\
  \hline
0.9&0.252883& 0.254189& 0.23152&  0.232226& 0.209792& 0.210126& 0.187339&
   0.187457& 0.166926& 0.166935\\
   \hline
\end{tabular}
\end{table}
\begin{table}
\centering
\small\addtolength{\tabcolsep}{-1pt}
\caption{Approximate results of the control function $u({\zeta})$ for $\mu=0.6, 0.7, 0.8, 0.9$ and $ 0.99.$ }
\begin{tabular}{|c|c|c|c|c|c|c|c|c|c|c|c|c|c|c|} \hline
${\zeta}$ &  \multicolumn{2}{|c|}{$\mu=0.6$} &\multicolumn{2}{|c|}{$\mu=0.7$} &\multicolumn{2}{|c|}{$\mu=0.8$}& \multicolumn{2}{|c|}{$\mu=0.9$} &\multicolumn{2}{|c|}{$\mu=0.99$}\\[0.5ex] \cline{2-11}
& OBWs &  FBWs & OBWs &  FBWs &  OBWs &  FBWs &  OBWs &  FBWs &  OBWs &  FBWs \\ \cline{2-11}
\hline
0.1&-0.394877& -0.439851& -0.405583& -0.441411& -0.415234& -0.436168&
-0.424029& -0.428269& -0.431123& -0.432319\\
\hline

0.2&-0.352463& -0.439495& -0.362216& -0.416708& -0.370041& -0.398336&
-0.376184& -0.385921& -0.380349& -0.381092\\
\hline
0.3&-0.313901& -0.381998& -0.321505& -0.358933& -0.326727& -0.343064&
-0.329709&-0.334278& -0.330532& -0.33066\\

\hline
 0.4&-0.288216& -0.267363& -0.281435& -0.268088& -0.284011& -0.270352&-0.283795& -0.27334& -0.28157& -0.281023\\
\hline

0.5&-0.265345& -0.21654& -0.25619& -0.205936& -0.247159& -0.205842&
-0.239332& -0.209313& -0.233211& -0.232294\\
\hline
 0.6&-0.233613& -0.182516& -0.222844& -0.182709& -0.210499& -0.180697&
-0.197964& -0.176436& -0.187214& -0.186489\\
 \hline
 0.7&-0.194706& -0.150884&-0.182512& -0.15294& -0.168483& -0.147376&
-0.15382& -0.137598& -0.140986& -0.140543\\
\hline
 0.8&-0.149777& -0.121644& -0.135998& -0.116629& -0.121471& -0.105878&
-0.106944& -0.0927991& -0.0945155& -0.094456\\
\hline
 0.9&-0.0996635& -0.0947954& -0.0839105& -0.0737776& -0.0697543&
-0.0562032& -0.057378&-0.042039& -0.0477968& -0.0482282\\
\hline
 
\end{tabular}
\end{table}
\begin{table}[h]
\begin{center}
\caption{Absolute errors of the state  and control functions for ${\mu}=1$.} 
\begin{tabular}{|c|c|c|c|c|c|c|} \hline
${\zeta}$ &  \multicolumn{3}{|c|}{$ k=3, M=3 $}  &\multicolumn{3}{|c|}{$ k=3, M=7 $}\\[0.5ex] \cline{2-7}
& $E_{x_{1}({\zeta})}$ & $E_{x_{2}({\zeta})}$ & $E_{u({\zeta})}$ & $E_{x_{1}({\zeta})}$ & $E_{x_{2}({\zeta})}$ & $E_{u({\zeta})}$\\
\hline
0.1&$2.08\times10^{-4}$& $7.11\times10^{-5}$& $1.27\times10^{-5}$& $1.6872\times10^{-9}$& $4.16\times 10^{-5}$&
 $ 2.008\times10^{-6}$\\
 \hline
0.3&$1.85\times10^{-4}$& $5.06\times10^{-5}$&$1.35\times10^{-6}$& $2.01\times10^{-9}$&$3.13\times10^{-5}$& $1.53\times10^{-6}$\\
 \hline
0.2& $2.79\times 10^{-4}$ &$1.90\times10^{-5}$ &$2.34\times10^{-5}$& $2.98\times10^{-9}$& $3.61\times 10^{-5}$&
  $1.75\times10^{-6}$\\
\hline  
 0.4&$1.49\times10^{-4}$&$9.50\times10^{-5}$& $2.72\times10^{-6}$&$7.66\times10^{-10}$ 
 &$2.72\times10^{-5}$&$1.34\times10^{-6}$\\
 \hline
0.5&$3\times10^{-4}$&$1.11\times10^{-5}$& $4.51\times10^{-5}$& $6.74\times10^{-9}$&$2.36\times10^{-5}$& 
  $1.18\times10^{-6}$\\
  \hline
0.6&$7.64\times10^{-5}$&$2.92\times10^{-5}$&$1.56\times10^{-5}$& $6.20\times10^{-10}$&
$2.04\times10^{-5}$&$ 1.046\times10^{-6}$\\
\hline
0.7&
 $1.02\times10^{-4}$&$ 4.62\times10^{-5}$&$1.69\times10^{-5}$& $1.09\times10^{-9}$&
 $1.77 \times 10^{-5}$ & $ 9.30\times10^{-7}$\\
 \hline
0.8&$6.81\times10^{-5}$& $3.61\times10^{-5}$&$2.71\times10^{-5}$& $7.40\times10^{-10}$& 
  $1.53\times10^{-5}$&$8.31\times10^{-7}$\\
  \hline
0.9&$5.51\times10^{-5}$&$ 2.33\times10^{-5}$&$1.91\times10^{-5}$& $2.81\times10^{-10}$& 
 $1.33\times10^{-5}$& $7.49\times10^{-7}$
\\
\hline

\end{tabular}

\footnotesize{$E_{x_{1}({\zeta})}$: Absolute error of the $x_{1}({\zeta})$; $E_{ x_{2}({\zeta})}$: Absolute error of the $x_{2}({\zeta})$; $ E_{u({\zeta})}$: Absolute error of the $ u({\zeta})$.}
\end{center}
\end{table}
\end{landscape}

\newpage

\begin{problem}
Let the following Caputo fractional optimal control of a Spring-Mass-Viscodamper system \textnormal{\cite{dehestani2022numerical}}:
\begin{align*}\nonumber
\min\Tilde{J}=\dfrac{1}{2}\int_{0}^{1}\left(x_{1}^2({\zeta})+x_{2}^2({\zeta})+u^2({\zeta})\right)d{\zeta},
\end{align*}
$$^{C}\mathcal{D}_{{\zeta}}^{{\mu}}x_{1}({\zeta})=x_{2}({\zeta}),$$
$$^{C}\mathcal{D}_{{\zeta}}^{{\mu}}x_{1}({\zeta})+ ^{C}\mathcal{D}_{{\zeta}}^{{\mu}}x_{2}({\zeta})=-x_{1}({\zeta})+u({\zeta}),$$
$$x_{1}(0)=1, x_{2}(0)=0.  $$
\end{problem}

This real-life multi-dimensional FOCP has been solved by the proposed methods using OBWs and FBWs with $k=2$ and $M=4$. Tables 6, 7, and 8 show the approximate results by using OBWs and FBWs methods of $x_{1}(\zeta)$, $x_{2}(\zeta)$ and $u(\zeta)$, respectively, for different values of $\mu$. Table 9 shows the approximate values of the cost function using OBWs and FBWs. Fig. 5 shows the approximate results of the state and control functions for different values of $\mu$ using the OBWs method. Fig. 6 shows the approximate results of the state and control functions for different values of $\mu$ using the FBWs method. It has been observed that the FBWs gives  better results than the OBWs method. 

\begin{table}[H]
\centering
\small\addtolength{\tabcolsep}{-1pt}
\caption{Approximate results of the state function $x_{1}({\zeta})$ for $\mu= 0.7, 0.8, 0.9$ and $ 0.99.$ }
\begin{tabular}{|c|c|c|c|c|c|c|c|c|c|c|c|c|} \hline
${\zeta}$ &\multicolumn{2}{|c|}{$\mu=0.7$} &\multicolumn{2}{|c|}{$\mu=0.8$}& \multicolumn{2}{|c|}{$\mu=0.9$} &\multicolumn{2}{|c|}{$\mu=0.99$}\\[0.5ex] \cline{2-9}
& OBWs &  FBWs & OBWs &  FBWs &  OBWs &  FBWs &  OBWs &  FBWs \\ \cline{2-9}
\hline
0.1&0.966394& 0.966394& 0.981324& 0.980995& 0.989744& 0.989744& 
  0.994214& 0.994204\\
\hline
0.2&0.923771& 0.923771& 0.948154& 0.948506& 0.966845& 0.966845& 
  0.978574& 0.978584\\
\hline
 0.3&0.876228& 0.876228& 0.908281& 0.908386& 0.935634& 0.935634& 
  0.954804& 0.954809\\
\hline
 0.4&0.82555& 0.82555& 0.865355& 0.864927& 0.898809& 0.898809& 0.924446& 
  0.924435\\
\hline
0.5& 0.780842& 0.780842& 0.819575& 0.819628& 0.857818& 0.857818& 
  0.888942& 0.888951\\
\hline
 0.6&0.739379& 0.739379& 0.775472& 0.775207& 0.814731& 0.814731& 
  0.849126& 0.849126\\
\hline
 0.7&0.700506& 0.700506& 0.732208& 0.731885& 0.770292& 0.770292&
  0.806145& 0.806144\\
\hline
 0.8&0.663781& 0.663781& 0.68991& 0.689652& 0.724948& 0.724948& 0.760532&
   0.760533\\
\hline
0.9&0.628884& 0.628884& 0.648715& 0.648496& 0.679117& 0.679117& 
  0.712818& 0.712818\\
\hline
 
\end{tabular}
\end{table}

\begin{table}[h]
\centering
\small\addtolength{\tabcolsep}{-1pt}
\caption{Approximate results of the state function $x_{2}({\zeta})$ for $\mu=0.7, 0.8, 0.9$ and $ 0.99.$ }
\begin{tabular}{|c|c|c|c|c|c|c|c|c|c|c|c|c|} \hline
${\zeta}$ & \multicolumn{2}{|c|} {$\mu=0.7$}& \multicolumn{2}{|c|}{$\mu=0.8$}& \multicolumn{2}{|c|}{$\mu=0.9$} &\multicolumn{2}{|c|}{$\mu=0.99$}\\[0.5ex] \cline{2-9 }
& OBWs &  FBWs & OBWs &  FBWs &  OBWs &  FBWs &  OBWs &  FBWs  \\ \cline{2-9}
\hline
0.1&-0.215705& -0.215705& -0.172919& -0.174955& -0.136316& -0.137135&
-0.107702& -0.107768\\
\hline
 0.2&-0.306565& -0.306565& -0.272313& -0.270388& -0.231912& -0.231105&
-0.196345&-0.196281\\
\hline
 0.3&-0.358994& -0.358994& -0.333866& -0.333714& -0.301393& -0.301298&
-0.268831& -0.268818\\
 \hline
 0.4&-0.394773& -0.394773& -0.3736& -0.374958& -0.352656& -0.353445&
-0.327588& -0.327655\\
 \hline
0.5&-0.405387& -0.405387& -0.404457& -0.404297& -0.392381& -0.392322&
-0.37496& -0.37495
 \\
 \hline
 0.6&-0.41188& -0.41188& -0.421043& -0.420716& -0.420097& -0.420075&
-0.412506& -0.412508\\
\hline
0.7&-0.414907& -0.414907& -0.430212& -0.429804& -0.439822& -0.439739&
-0.44231& -0.442308\\
\hline
0.8&-0.41493& -0.41493& -0.434843& -0.434674& -0.453871& -0.453822&
-0.466024& -0.466022\\
 \hline
0.9&-0.412291& -0.412291&-0.437815& -0.437907& -0.464558& -0.464583&             
-0.485299& -0.485301\\
\hline
\end{tabular}
\end{table}
\begin{table}[h]
\centering
\small\addtolength{\tabcolsep}{-1pt}
\caption{Approximate results of the control function $u({\zeta})$ for $\mu= 0.7, 0.8, 0.9$ and $ 0.99.$ }
\begin{tabular}{|c|c|c|c|c|c|c|c|c|c|c|c|c|} \hline
${\zeta}$ & \multicolumn{2}{|c|}{$\mu=0.7$} &\multicolumn{2}{|c|}{$\mu=0.8$}& \multicolumn{2}{|c|}{$\mu=0.9$} &\multicolumn{2}{|c|}{$\mu=0.99$}\\[0.5ex] \cline{2-9}
& OBWs &  FBWs & OBWs &  FBWs &  OBWs &  FBWs &  OBWs &  FBWs  \\ \cline{2-9}
\hline
0.1&-0.140779& -0.140779& -0.123588& -0.12401& -0.0989385& -0.0991709&
-0.0726188& -0.0726563\\
\hline
0.2&-0.0852573& -0.0852573& -0.0636988& -0.0634089& -0.0386612&
-0.0384181& -0.0140141& -0.0139818\\
\hline
 0.3&-0.0384791& -0.0384791& -0.0149009& -0.0151017& 0.0102628& 0.010171&
   0.0322697& 0.0322735\\
\hline
 0.4&-0.0119472& -0.0119472& 0.0232195& 0.0231515& 0.0475384& 0.0473008& 
  0.0661279& 0.0660906\\
 \hline
 0.5&0.031522& 0.031522& 0.0509791& 0.0540807& 0.073211& 0.0739984&
  0.0874771& 0.0874875\\
\hline
 0.6&0.0572762& 0.0572762& 0.0713499& 0.070753& 0.0871447& 0.087106& 
  0.0960953& 0.0960946\\
 \hline
 0.7&0.0681257& 0.0681257& 0.0822755& 0.0811974& 0.0899385& 0.0897351& 
  0.09208& 0.0920792\\
\hline
0.8&0.0660471& 0.0660471& 0.0782809& 0.0784622& 0.0786118& 0.0786288& 
  0.0750283& 0.0750288\\
\hline
 0.9&0.0525099& 0.0525099& 0.0546616& 0.0555952& 0.0503521& 0.0505307& 
  0.0445306& 0.0445311\\
\hline
 
\end{tabular}
\end{table}
\begin{table}[h]
\centering
\caption{Approximate values of the cost function $\Tilde{{J}}$ for different values of  $\mu$.}
\begin{tabular}{|c|c|c|}\hline
$\mu $& OBWs     & FBWs \\ \cline{1-3}
\hline
1 & 0.454499 & 0.454499\\
\hline
0.99 & 0.452568 & 0.452568\\
\hline
0.9 & 0.434207&  0.434201 \\
\hline
0.8 & 0.412561 &0.412541\\
\hline
0.7 & 0.391139 & 0.391116\\
\hline
 \end{tabular}
\end{table}

\begin{figure}
\centering
\begin{subfigure}{0.4\textwidth}
    \includegraphics[width=\textwidth]{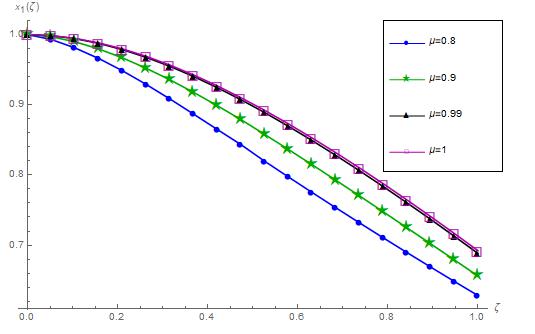}
    \caption{Approximate and exact results of $x_{1}(\zeta)$.}
    \label{5(a)p2}
\end{subfigure}
\hfill
\begin{subfigure}{0.4\textwidth}
    \includegraphics[width=\textwidth]{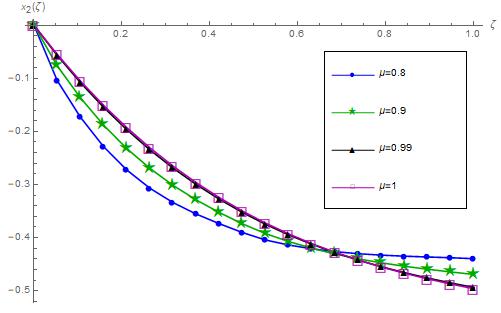}
    \caption{Approximate and exact results of  $x_{2}(\zeta)$.}
    \label{5(b)p2}
\end{subfigure}
\hfill
\begin{subfigure}{0.4\textwidth}
    \includegraphics[width=\textwidth]{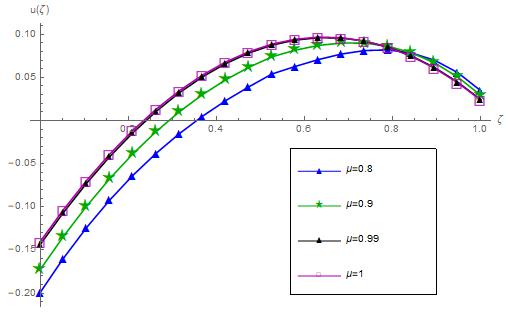}
    \caption{Approximate and exact results of $u(\zeta)$.}
    \label{5(c)p2}
\end{subfigure}
        
\caption{OBWs results when $k=2$ and $M=4$.}
\label{5p2}
\end{figure}
\begin{figure}
\centering
\begin{subfigure}{0.4\textwidth}
    \includegraphics[width=\textwidth]{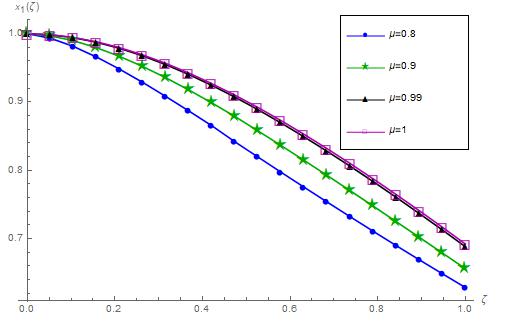}
    \caption{Approximate and exact results of $x_{1}(\zeta)$.}
    \label{6(a)p2}
\end{subfigure}
\hfill
\begin{subfigure}{0.4\textwidth}
    \includegraphics[width=\textwidth]{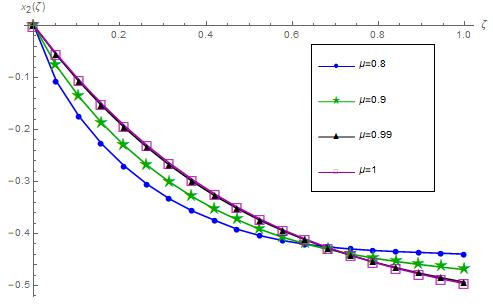}
    \caption{Approximate and exact results of $x_{1}(\zeta)$.}
    \label{6(b)p2}
\end{subfigure}
\hfill
\begin{subfigure}{0.4\textwidth}
    \includegraphics[width=\textwidth]{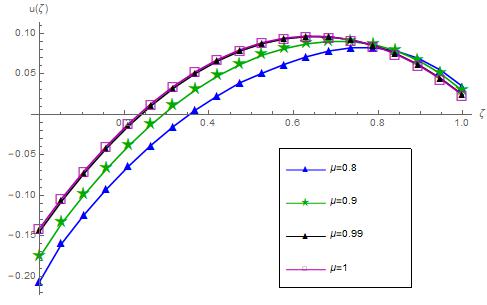}
    \caption{Approximate and exact results of $u(\zeta)$.}
    \label{6(c)p2}
\end{subfigure}
        
\caption{FBWs results when $k=2$ and $ M=4$.}
\label{6p2}
\end{figure}
\newpage

\section{Conclusion}
This paper aims to describe a procedure for computing the numerical result of multi-dimensional fractional optimal control problems. In this work, OBWs and FBWs methods have been used to solve multi-dimensional fractional optimal control problems. As mentioned, the proposed methods transform the multi-dimensional fractional optimal problem into a system of equations that can be solved using the Lagrange multiplier technique. The numerical method based on fractional wavelets is a relatively new study area. Fractional wavelets are piecewise and continuous functions with compact support $[0, 1]$. The novel approach obtained in this paper to solving multi-dimensional fractional optimal control problems with the help of operational matrices is based on using  FBWs. Implementing the proposed method is very simple and effective for solving multi-dimensional fractional optimal control problems. So, the experimental results from the test problems show that the proposed numerical techniques are accurate with high accuracy and require less computation overhead. It has been observed that the FBWs method gives better results than the OBWs method. Furthermore, error estimation and convergence analysis for the proposed numerical approach have also been established.\\\\

\section*{Declarations}

\subsection*{Ethical Approval }
Not applicable.
\subsection*{Competing interests }
The authors declare they have no competing interests in this manuscript.
\subsection*{Authors' contributions}
Akanksha Singh: Writing-Original draft, Conceptualization, Methodology, Investigation,. S. Saha Ray: Supervision, Writing-review  and editing.
\subsection*{Funding}
Not applicable.

\subsection*{Availability of data and materials }
No data were utilised in the study described in the article.

\subsection*{Acknowledgement}
The ``University Grants Commission (UGC)" fellowship scheme provided financial support with NTA Ref. No. 201610127052 is gratefully acknowledged by the first author. \\ 
\\

\bibliographystyle{apacite}
\bibliography{References}
 \end{document}